\documentclass{article}

\usepackage{datetime2} 
\usepackage[margin=1in]{geometry}
\usepackage[utf8]{inputenc}
\usepackage{amsmath}
\usepackage{bm}
\usepackage{verbatim}
\usepackage{float}
\usepackage{graphicx}
\usepackage{amssymb}

\usepackage{bbm}
\usepackage{romanbar}
\usepackage{amsthm}
\usepackage{enumerate}

\newtheorem{thm}{Theorem}[section]
\newtheorem{cor}[thm]{Corollary}
\newtheorem{prop}[thm]{Proposition}
\newtheorem{assum}{Assumption}
\theoremstyle{definition}
\newtheorem{definition}[thm]{Definition}

\theoremstyle{remark}
\newtheorem{rem}[thm]{Remark}
\numberwithin{equation}{section}

\title{Estimation of sub-Gaussian random vectors using the method of moments.}
\author{Taras Bodnar, Dmitry Otryakhin$^1$ and Erik Thorsen}
\date{%
	Department of Mathematics, Stockholm University\\%
	\vspace*{2pt}
	$^1$ d.otryakhin.acad@protonmail.ch\\%
	\vspace*{3pt}
	\today
}

\usepackage[square,numbers]{natbib}

\begin{document}
\maketitle

\begin{abstract}
	The sub-Gaussian stable distribution is a heavy-tailed elliptically contoured law which has interesting applications in signal processing and financial mathematics. This work addresses the problem of feasible estimation of this type of distributions. We present a method based on application of the method of moments to the empirical characteristic function. Further, we show almost sure convergence of our estimators, discover their limiting distribution and demonstrate their finite-sample performance.
\end{abstract}

\section{Introduction}
The sub-Gaussian stable distribution can be defined by its characteristic function.

Let $\mathbf{x}\sim S_p(\boldsymbol\mu, \boldsymbol\Sigma, \boldsymbol{\alpha})$ where $S_p()$ denotes the $p$ dimensional sub-Gaussian stable distribution. We have that the characteristic function (Prop. 2.5.2, \cite{Samorod}) is
\begin{equation}\label{eqn:charac}
	\phi(\mathbf{t}) = \mathbb{E}[e^{i\mathbf{t}^\top \mathbf{x}}] = e^{i \mathbf{t} ^\top \boldsymbol\mu - \left(\frac{1}{2} \mathbf{t} ^\top \boldsymbol\Sigma \mathbf{t} \right)^{\alpha/2}}
\end{equation}
where $0<\alpha<2$, $\mathbf{t} \in \mathbb{R}^p$, $\boldsymbol\mu \in  \mathbb{R}^p$, and $\boldsymbol\Sigma$ is a $p \times p$ symmetric positive semi-definite matrix. By default, column-type vectors are used throughout the article.

$S_p$ is a generalization of the one-dimensional symmetric stable distribution. As seen from (\ref{eqn:charac}), compared to the general multivariate stable distribution, the sub-Gaussian stable distribution is simpler: its characteristic function does not require one to work with a complicated spectral measure. It possesses interesting properties: the law belongs to the class of elliptically contoured stable distributions and has heavy tails. These properties make it a useful tool for applications in finance \cite{nolan99}.

The univariate stable distribution was found to be useful in video foreground detection. \cite{ming03} describes a method for detection of moving objects on images produced by static cameras. It used Cauchy distribution to model ratios of pixel intensities. Later, \cite{bhaskar2010} presented a similar model, but instead of Cauchy distributed random variables they used mixtures of univariate stable ones. The distributions of ratios of pixel intensities are observed to be symmetric. Moreover, it is reasonable to consider joint distributions of intensities of different colours. Thus, the sub-Gaussian stable distribution  naturally arises in this field.

The first ideas of applying the method of moments to the empirical characteristic function for estimation of distributions defined by (\ref{eqn:charac}) originate in \cite{S_J_Press_77}. They were later enhanced in \cite{kring2009estimation}, though neither of the articles provided a complete feasible estimation strategy. In this work, we use some of the ideas presented in those two works in order to create three efficient estimators: one for parameter $\alpha$, one for diagonal elements of $\boldsymbol\Sigma$, and one for non-diagonal elements of $\boldsymbol\Sigma$. In the past decade there have been created two estimation methods of other kinds for the model (\ref{eqn:charac}): in \cite{omelch14} the author describes a maximum likelyhood method while \cite{EM_SGest18} presents an EM algorithm.

\section{Estimation based on the empirical characteristic function}

Assume that we have a sample of independent identically distributed random variables $\mathbf{x}_j,\; j=1,2,3,...,n$ whose entries are distributed according to the sub-Gaussian stable distribution. For the given sample at hand, the empirical characteristic function is defined as
\begin{equation}
\label{eq:phi_empiric}
    \hat\phi_n(\mathbf{t}):= \frac{1}{n}\sum_{j=1}^n e^{i\mathbf{t}^\top \mathbf{x}_j}
\end{equation}
Through the result of the law of large numbers we have that for any given $\mathbf{t} \in \mathbb{R}^p$ \begin{equation}\label{eqn:empir_charac}
	\hat\phi_n(\mathbf{t}) \stackrel{a.s.}{\rightarrow} \phi(\mathbf{t}),
\end{equation}
where $|z|$ denotes the absolute value of the complex variable $z$: $|z|=\sqrt{\operatorname{Re}(z)^2 + \operatorname{Im}(z)^2}$.

\begin{assum}
	\label{ass:e_Sig_e_n_0}
	$\mathbf{e}^T \boldsymbol\Sigma \mathbf{e} \neq 0$.
\end{assum}

We first write an estimator of $\alpha$ developed by S. J. Press (formula (4.4) in \cite{S_J_Press_77}). Under Assumption \ref{ass:e_Sig_e_n_0}
\begin{equation}\label{eqn:alpha_hat}
    \hat \alpha^p = \frac{\operatorname{log}\left[ \frac{\operatorname{log}|\hat \phi_n(s_1\mathbf{e})|}{\operatorname{log}|\hat \phi_n(s_2\mathbf{e})|} \right]}{\operatorname{log}|s_1/s_2|} = \frac{1}{\operatorname{log}|s_1/s_2|} \log \left[\frac{
    \log \left(\operatorname{Re}^2[\hat \phi_n(s_1\mathbf{e})] + \operatorname{Im}^2[\hat \phi_n(s_1\mathbf{e})] \right)}{\log \left(\operatorname{Re}^2[\hat \phi_n(s_2\mathbf{e})] + \operatorname{Im}^2[\hat \phi_n(s_2\mathbf{e})]\right)} \right],
\end{equation}
where $\mathbf{e}=(1,1,\dots,1)$, and $s_1 \neq s_2$ are scalars.
It is a simple estimator, but it works well only if all components of the vectors have comparable narrow distributions. There are reasons why we don't rely on the  estimator of $\alpha$ (\ref{eqn:alpha_hat}). First, the estimator exploits a  cumulative statistic $\mathbf{t}^T \mathbf{x}$ which has larger variance than individual summands $\mathbf{t}_j^T\mathbf{x}_j$. Moreover, one component may dominate all the others reducing the data efficiently used for estimation. Second, any estimate $\hat \alpha$ converging to the true value inside the interval $(0,2)$ has  variance decreasing with $n$, while the variance of $\mathbf{t}^T \mathbf{x}$ increases with $n$. Thus, it seems to be a good idea to estimate $\alpha$ from the components of $\mathbf{t}$ separately and then to aggregate them. We introduce an auxiliary estimator $\hat \alpha_s(k)$ based on only 1 component of the sub-Gaussian vector
\begin{equation}
\label{def:alpha_s}
\hat \alpha_s(k) := \frac{\operatorname{log}\left[ \frac{\operatorname{log}|\hat \phi_n(s_1\mathbf{e}_k)|}{\operatorname{log}|\hat \phi_n(s_2\mathbf{e}_k)|} \right]}{\operatorname{log}|s_1/s_2|}
\end{equation}
where $e_k$ has 1 only in the k-s component and every other one equals zero. Then the aggregation is done via simple averaging:
\begin{equation}
\label{def:alpha_m}
\hat \alpha^{mult}:=\frac{1}{p}\sum_{k=1}^p \hat \alpha_s(k)
\end{equation}
\begin{rem}
	Pre-estimation of the diagonal elements of $\boldsymbol\Sigma$ with subsequent rescaling should help equalize influence of different components of $\mathbf{x}$ on the estimate so that unweighted average (\ref{def:alpha_m}) is not affected by highly-dispersed components.
\end{rem}
\begin{assum}
\label{ass:Sigma_d_n_0}
    $\boldsymbol\Sigma_{kk} \neq 0$ for all $k=1, \dots, p$. That is, no component of $\mathbf{x}$ in (\ref{eqn:charac}) has a degenerate distribution.
\end{assum}
\begin{prop}[Strong consistency of the alpha estimators]
	\label{prop:alpha_consistency}	
	Suppose that Assumption \ref{ass:Sigma_d_n_0} holds. Set $s_1>0$ and $s_2>0$ such that $s_1 \neq s_2$. Then for $\hat \alpha_s$ defined in (\ref{def:alpha_s}) and $\hat \alpha^{mult}$ defined in (\ref{def:alpha_m}) it holds that 
	\begin{equation}
		\label{eqn:alpha_consistency}
        \hat \alpha_s(k) \stackrel{a.s.}{\rightarrow} \alpha \qquad \hat \alpha^{mult} \stackrel{a.s.}{\rightarrow} \alpha
	\end{equation}
\end{prop}	
\begin{proof}
    From (\ref{eqn:empir_charac})  we have 
    \begin{equation}
        |\hat \phi_n(s_j\mathbf{e}_k) | \stackrel{a.s.}{\rightarrow} |\phi(s_j\mathbf{e}_k) | = \exp \left[-\left(\frac{1}{2} s_j \mathbf{e}_k^T \boldsymbol\Sigma s_j \mathbf{e}_k\right)^{\alpha/2} \right] = \exp \left[-\left( \frac{1}{2} s_j^2 \boldsymbol\Sigma_{kk} \right)^{\alpha/2} \right], \quad j=1,2
    \end{equation}
    Continuous mapping theorem yields
    \begin{equation}
    \label{eq:SCAE_log_phi_conv}
        \log|\hat \phi_n(s_j\mathbf{e}_k) | \stackrel{a.s.}{\rightarrow} -\left( \frac{1}{2} s_j^2 \boldsymbol\Sigma_{kk} \right)^{\alpha/2}, \quad j=1,2 
    \end{equation}  
    Function $f(\delta_1, \delta_2)=\frac{\delta_1}{\delta_2}$ is discontinuous at only one point--- $\delta_2=0$. Hence, it is possible to apply Continuous mapping theorem with $f$ to (\ref{eq:SCAE_log_phi_conv}) since  
    \begin{align*}
    		\delta_2=0 \iff \log|\hat \phi_n(s_2\mathbf{e}_k)| = 0 \quad \text{is a zero-probability event. Then, due to Assumption \ref{ass:Sigma_d_n_0}}
    \end{align*}
    
    \begin{equation}
        \qquad \frac{\log|\hat \phi_n(s_1\mathbf{e}_k)|}{\log|\hat \phi_n(s_2\mathbf{e}_k) |} \stackrel{a.s.}{\rightarrow}  \left(\frac{s_1}{s_2}\right)^{\alpha}
    \end{equation}
    \begin{equation}
        \log \frac{\log|\hat \phi_n(s_1\mathbf{e}_k)|}{\log|\hat \phi_n(s_2\mathbf{e}_k) |} \stackrel{a.s.}{\rightarrow}  \alpha \cdot \log\left(\frac{s_1}{s_2}\right)
    \end{equation}
    Thus, the first identity in (\ref{eqn:alpha_consistency}) is proven; the second one follows immediately.
\end{proof}	
\begin{prop}[Strongly consistent matrix estimators]
	\label{prop:estimators}	
	Consider the distribution determined by (\ref{eqn:charac}) and let $\hat \alpha$ be a strongly consistent estimator of parameter $\alpha$ so that $\hat \alpha \stackrel{a.s.}{\rightarrow} \alpha$. Then
\begin{align}\label{eqn:matrix_estims}
\begin{split}
    [\hat{\boldsymbol{\Sigma}}]_{ii} &:= \frac{2}{{s^2_i}}(-\operatorname{log}|\hat \phi_n(s_i\mathbf{e}_i)|)^{2/\hat \alpha} \stackrel{a.s.}{\rightarrow} \boldsymbol\Sigma_{ii}, \quad i \leq p \\
    [\hat{\boldsymbol{\Sigma}}]_{ij} &:= \frac{\big[-\operatorname{log}(|\hat \phi_n(\mathbf{e}_i+\mathbf{e}_j)| )\big]^{2/\hat \alpha} - \big[-\operatorname{log}( |\hat \phi_n(\mathbf{e}_i-\mathbf{e}_j)| )\big]^{2/\hat \alpha}}{2}\stackrel{a.s.}{\rightarrow} \boldsymbol\Sigma_{ij}, \quad i \leq p, \quad j \leq i
\end{split}
\end{align}
\end{prop}

\begin{proof}
(\ref{eqn:charac}) and (\ref{eqn:empir_charac}) yield the following convergence:
\begin{equation}
\label{eq:log_abs}
    \log|\hat \phi_n(\mathbf{t}) | \stackrel{a.s.}{\rightarrow} - \left(\frac{1}{2}\mathbf{t}^\top \boldsymbol\Sigma \mathbf{t}\right)^{\frac{\alpha}{2}}.
\end{equation}
Using (\ref{eq:log_abs}), condition $\hat \alpha \stackrel{a.s.}{\rightarrow} \alpha$ and the multivariate continuous mapping theorem, the following statement is obtained:
\begin{equation}
\label{eq:log_abs_alpha_hat}
\big[-\log|\hat \phi_n(\mathbf{t})|\big]^{2/{\hat \alpha}} \stackrel{a.s.}{\rightarrow}  \frac{1}{2}\mathbf{t}^\top \boldsymbol\Sigma \mathbf{t}.
\end{equation}
We consider three types of vector $\mathbf{t}$ used in \cite{kring2009estimation}, (e.g. Corollary 1): $\mathbf{t}=s_i \mathbf{e}_i$, $\mathbf{t}= \mathbf{e}_i+\mathbf{e}_j$ and $\mathbf{t}=\mathbf{e}_i-\mathbf{e}_j$. 

a$)$ Since $(s_i\mathbf{e}_i)^T \boldsymbol\Sigma (s_i\mathbf{e}_i)=s_i^2\boldsymbol\Sigma_{ii}$, for all $i=1,...,p$ we may construct an estimator for the diagonal elements of $\boldsymbol\Sigma$ by substituting $\mathbf{t}$ in (\ref{eq:log_abs_alpha_hat}) by $s_i\mathbf{e}_i$.

b$)$ Given the scheme in Figure \ref{vector-matrix-vector_multiplication}
	\begin{figure}[h!]
		\begin{centering}
	    \includegraphics[keepaspectratio, height=5.3cm]{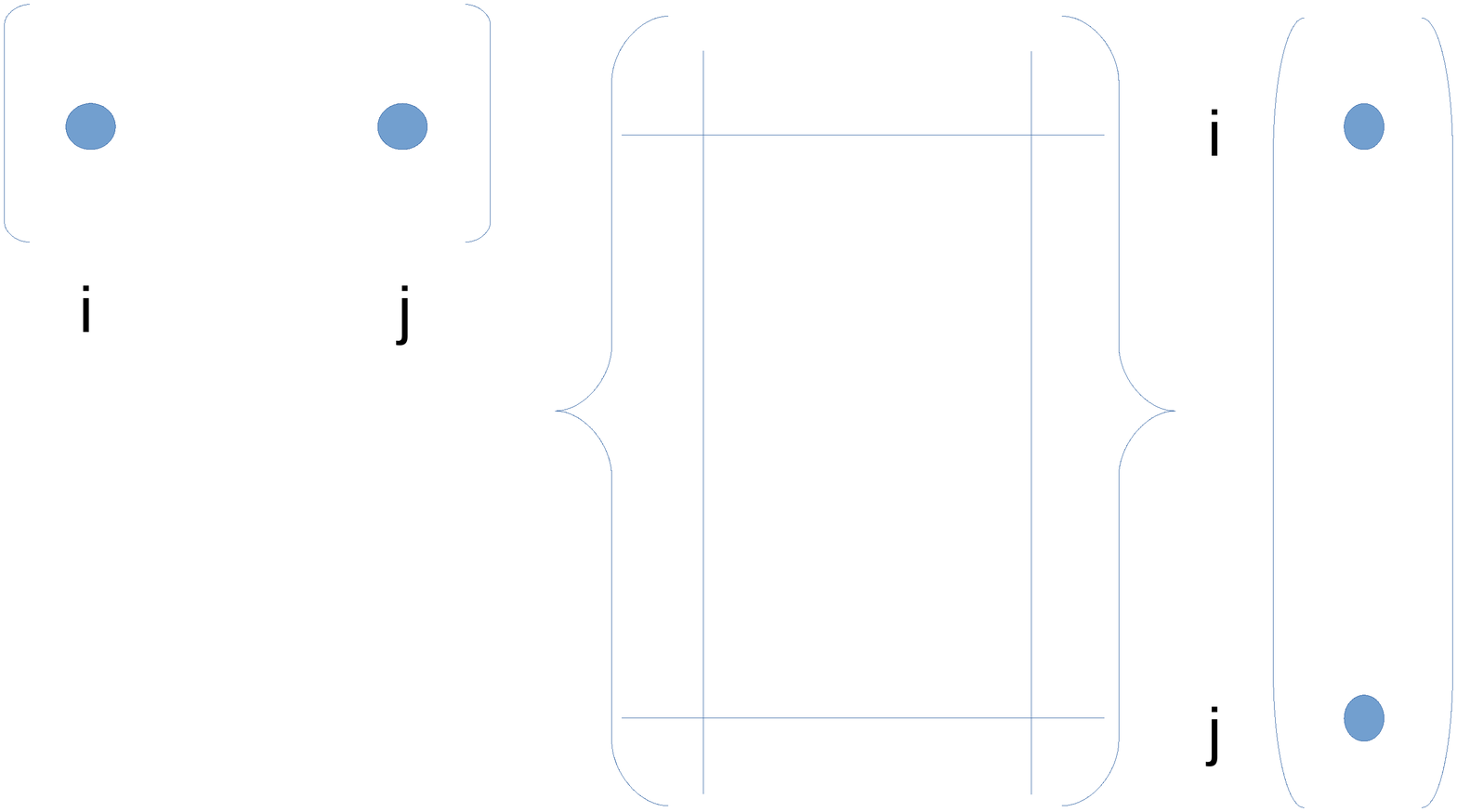}
	    \caption{A scheme of the matrix multiplication.}
	    \label{vector-matrix-vector_multiplication}
        \end{centering}
    \end{figure}
it is clear that 
\begin{align}
\begin{split}
	(\mathbf{e}_i-\mathbf{e}_j)^T\boldsymbol\Sigma(\mathbf{e}_i-\mathbf{e}_j) &= (\boldsymbol\Sigma_{ii} - \boldsymbol\Sigma_{ij}) - (\boldsymbol\Sigma_{ji} - \boldsymbol\Sigma_{jj}) = \boldsymbol\Sigma_{ii} - 2\boldsymbol\Sigma_{ij} + \boldsymbol\Sigma_{jj}
	\\
	(\mathbf{e}_i+\mathbf{e}_j)^T\boldsymbol\Sigma(\mathbf{e}_i+\mathbf{e}_j) &= \boldsymbol\Sigma_{ii} + 2\boldsymbol\Sigma_{ij} + \boldsymbol\Sigma_{jj} \quad \text{analogously.}
\end{split}
\end{align}
The last two statements yield that 
\begin{equation}
	\label{eq:Sigma_non-diagonal}
	(\mathbf{e}_i+\mathbf{e}_j)^T\boldsymbol\Sigma(\mathbf{e}_i+\mathbf{e}_j) - (\mathbf{e}_i-\mathbf{e}_j)^T\boldsymbol\Sigma(\mathbf{e}_i-\mathbf{e}_j) = 4\boldsymbol\Sigma_{ij}.
\end{equation}
At last, using (\ref{eq:log_abs_alpha_hat}) and (\ref{eq:Sigma_non-diagonal}) we obtain the second formula in (\ref{eqn:matrix_estims})
\end{proof}

Note, that in this proposition, $\hat \alpha$ is any plug-in estimator, although in this work, we confine ourselves to using (\ref{eqn:alpha_hat}) and (\ref{def:alpha_m}). Also, in Proposition \ref{prop:estimators} no assumptions are imposed on $\boldsymbol\Sigma$ except that it should be a legitimate covariance matrix. 

Let $\hat{\boldsymbol{\Sigma}}_d$ denote the vector of the diagonal elements of matrix $\hat{\boldsymbol{\Sigma}}$, $\hat {\boldsymbol{\Sigma}}_d[i]: = \hat{\boldsymbol{\Sigma}}[i,i]$ and let us construct $\hat {\boldsymbol{\Sigma}}_{nd}$--- the vector of the non-diagonal elements of $\hat{\boldsymbol{\Sigma}}$. In short, to do so, we take the elements under the diagonal row-by-row and stack them together, see Table \ref{tbl:indexing_nd}. 

\setcounter{table}{0}
\begin{table}[H]
	\begin{centering}
	\begin{tabular}{| c || c || c | c || c | c | c || c | c | c | c || c}
		\hline
		i  & 2 & 3 & 3 & 4 & 4 & 4 & 5 & 5 & 5 & 5 & $\dots$\\ \hline
		j  & 1 & 1 & 2 & 1 & 2 & 3 & 1 & 2 & 3 & 4 & $\dots$\\ \hline
		\# & 1 & 2 & 3 & 4 & 5 & 6 & 7 & 8 & 9 & 10 & $\dots$\\
		\hline
	\end{tabular}
    \caption{Packing the elements under the diagonal of a matrix into a vector. $i$ is the row index of the matrix $\boldsymbol\Sigma$, $j$ is the column index and \# is the vector index.}
    \label{tbl:indexing_nd}
    \end{centering}
\end{table}
Numbers of elements in the groups $i$ form an arithmetic progression with step one. Thus, \# corresponding to the maximal elements $j$ in each group is the sum of the progression, from which we obtain 
\begin{equation}
\label{eq:index_hash}
\#(i,j)=\frac{i(i-1)}{2} - (i-1) + j, \quad j \in [1, i-1].
\end{equation}
\begin{definition}
	Given a matrix $\boldsymbol\Sigma$, vector $\boldsymbol\Sigma_{nd}$ is defined through its components: 
	\begin{equation*}
	    \boldsymbol\Sigma_{nd}[\#(i,j)]:=\boldsymbol\Sigma[i,j],
	\end{equation*}
	where $\#(i,j)$ is defined via formula (\ref{eq:index_hash}).
\end{definition}	
\begin{cor}
	\label{cor:all_estim}
	For $\hat \alpha^p$ defined by (\ref{eqn:alpha_hat}), $\hat \alpha^{mult}$--- by (\ref{def:alpha_m}), and $\boldsymbol\Sigma_d$, $\boldsymbol\Sigma_{nd}$--- by (\ref{eqn:matrix_estims})
	\begin{align}
		\begin{split}
			&\begin{pmatrix}
				\hat\alpha^p \\
				\hat {\boldsymbol\Sigma}_d \\
				\hat {\boldsymbol\Sigma}_{nd}  
			\end{pmatrix} 
			\stackrel{a.s.}{\rightarrow}
			\begin{pmatrix}
				\alpha \\
				\boldsymbol\Sigma_d \\
				\boldsymbol\Sigma_{nd}    
			\end{pmatrix} 
			\quad \text{under Assumption \ref{ass:e_Sig_e_n_0} and} \\
			&\begin{pmatrix}
				\hat\alpha^{mult} \\
				\hat{\boldsymbol\Sigma}_d \\
				\hat{\boldsymbol\Sigma}_{nd}  
			\end{pmatrix} 
			\stackrel{a.s.}{\rightarrow}
			\begin{pmatrix}
				\alpha \\
				\boldsymbol\Sigma_d \\
				\boldsymbol\Sigma_{nd}    
			\end{pmatrix}  
			\quad \text{under Assumption \ref{ass:Sigma_d_n_0}}
		\end{split}
	\end{align}	
\end{cor}

\begin{proof}
	Follows from propositions \ref{prop:alpha_consistency} and \ref{prop:estimators} and properties of the strong convergence.
\end{proof}
While Assumption \ref{ass:Sigma_d_n_0} is easy to check and virtually always holds, Assumption \ref{ass:e_Sig_e_n_0} is more confusing. This seems to be an advantage of $\hat \alpha^{mult}$ over $\hat \alpha^p$.

In the end of this section, we introduce an estimator of $\boldsymbol\mu$ also based on the characteristic functions (\ref{eqn:charac}) and (\ref{eq:phi_empiric}). We write the following expression extracted from (\ref{eqn:charac}):
\begin{align}
\label{eq:im-to-re_phi}
    \frac{\operatorname{Im}(\phi(\mathbf{t}))}{\operatorname{Re}(\phi(\mathbf{t}))} = \tan\left[ \mathbf{t}^T \boldsymbol\mu \right]
\end{align}
If $\mathbf{t}^T \boldsymbol\mu < \pi/2$, then expression (\ref{eq:im-to-re_phi}) is easily invertable. Thus, if there exists vector $\mathbf{M}$ such that $\forall k$ $\boldsymbol\mu[k]<\pi \mathbf{M}[k] /2$, then $\hat{\boldsymbol\mu}$ defined by
\begin{align}
\label{def:mu_estimator}
    \hat{\boldsymbol\mu}[k] := \mathbf{M}[k] \cdot \arctan\left[\frac{\operatorname{Im}(\phi(\mathbf{e}_k/\mathbf{M}[k]))}{\operatorname{Re}(\phi(\mathbf{e}_k/\mathbf{M}[k]))}\right] , \text{ for all } k=1,\dots,p
\end{align}
is a consistent estimator of $\boldsymbol\mu$.

\section[Limiting distribution]{The limiting distribution of the estimators}

Through the application of the multivariate central limit theorem receive the following result:
\begin{prop}
	Let $\mathbf{t}_1, \mathbf{t}_2,..., \mathbf{t}_m$ be a grid of fixed points in the support of \eqref{eqn:charac}. Then
	\label{CLT_Re_Im_vect}	
	\begin{align}
	\label{feuer_CLT_parts_ecf}
	\sqrt{n}\begin{pmatrix} 
	\operatorname{Re}[\hat{\phi}_n(\mathbf{t}_1)] - \operatorname{Re}[\phi(\mathbf{t}_1)] \\
	\operatorname{Re}[\hat{\phi}_n(\mathbf{t}_2)] - \operatorname{Re}[\phi(\mathbf{t}_2)] \\
	\vdots \\
	\operatorname{Re}[\hat{\phi}_n(\mathbf{t}_m)] - \operatorname{Re}[\phi(\mathbf{t}_m)] \\
	\operatorname{Im}[\hat{\phi}_n(\mathbf{t}_1)] - \operatorname{Im}[\phi(\mathbf{t}_1)] \\
	\operatorname{Im}[\hat{\phi}_n(\mathbf{t}_2)] - \operatorname{Im}[\phi(\mathbf{t}_2)] \\
	\vdots \\
	\operatorname{Im}[\hat{\phi}_n(\mathbf{t}_m)] - \operatorname{Im}[\phi(\mathbf{t}_m)] \\
	\end{pmatrix} \stackrel{d}{\rightarrow}
	\mathbb{N}_{2m}(0, \mathbf{\Omega})
	\end{align}
	where the $\mathbf{\Omega}$ consists of four m-by-m smaller  matrices
	\begin{equation}
	\label{Omega_repres}
	\mathbf{\Omega} = \begin{pmatrix} 
	\mathbf{\Omega}_{Re} & \mathbf{\Omega}_{RI} \\
	\mathbf{\Omega}_{RI}^T & \mathbf{\Omega}_{Im}\\
	\end{pmatrix}
	\end{equation}
	\begin{align*} 
	[\mathbf{\Omega}_{Re}]_{jk} =& \operatorname{Cov}\left[ \rm{Re}(e^{i\mathbf{t}^T_jX}), \rm{Re}(e^{i\mathbf{t}^T_kX}) \right] =  
	\frac{1}{2} \rm{Re}(\phi(\mathbf{t}_j+\mathbf{t}_k)) + \frac{1}{2} \rm{Re}(\phi(\mathbf{t}_k-\mathbf{t}_j)) -
	\rm{Re}[\phi(\mathbf{t}_j)] \cdot  \rm{Re} \{\phi(\mathbf{t}_k)\}= \\
	&\frac{1}{2}  \cos \left(\mathbf{t}^T_j\boldsymbol\mu+\mathbf{t}^T_k\boldsymbol\mu \right) \cdot \exp\left[-\left\{\frac{1}{2}(\mathbf{t}_j+\mathbf{t}_k)^T \boldsymbol\Sigma (\mathbf{t}_j+\mathbf{t}_k)\right\}^{\alpha/2}\right] + \\ 
	&\frac{1}{2}  \cos \left(\mathbf{t}^T_k\boldsymbol\mu-\mathbf{t}^T_j\boldsymbol\mu \right) \cdot \exp\left[-\left\{ \frac{1}{2}(\mathbf{t}_k-\mathbf{t}_j)^T \boldsymbol\Sigma (\mathbf{t}_k-\mathbf{t}_j)\right\}^{\alpha/2}\right] - \\
	&\cos \left(\mathbf{t}^T_j\boldsymbol\mu\right) \cdot \cos \left(\mathbf{t}^T_k\boldsymbol\mu \right) \cdot \exp\left[-\left\{\frac{1}{2}\mathbf{t}_j^T \boldsymbol\Sigma \mathbf{t}_j\right\}^{\alpha/2}\right] \cdot  \exp\left[-\left\{\frac{1}{2}\mathbf{t}_k^T \boldsymbol\Sigma \mathbf{t}_k\right\}^{\alpha/2}\right] 
	\end{align*}
    \begin{align*}    
	[\mathbf{\Omega}_{Im}]_{jk} =& \operatorname{Cov}\left[ \rm{Im}(e^{i\mathbf{t}^T_jX}), \rm{Im}(e^{i\mathbf{t}^T_kX}) \right] =  
	\frac{1}{2} \rm{Re}(\phi(\mathbf{t}_j+\mathbf{t}_k)) - \frac{1}{2}  \rm{Re}(\phi(-\mathbf{t}_j+\mathbf{t}_k)) -
	\rm{Re}[\phi(\mathbf{t}_j)]  \cdot  \rm{Im} \{\phi(\mathbf{t}_k)\} \\
	&\frac{1}{2}  \cos \left(\mathbf{t}^T_j\boldsymbol\mu+\mathbf{t}^T_k\boldsymbol\mu \right) \cdot \exp\left[-\left\{\frac{1}{2} (\mathbf{t}_j+\mathbf{t}_k)^T \boldsymbol\Sigma (\mathbf{t}_j+\mathbf{t}_k)\right\}^{\alpha/2}\right] - \\ 
	&\frac{1}{2}  \cos \left(\mathbf{t}^T_k\boldsymbol\mu-\mathbf{t}^T_j\boldsymbol\mu \right) \cdot \exp\left[-\left\{\frac{1}{2} (\mathbf{t}_k-\mathbf{t}_j)^T \boldsymbol\Sigma (\mathbf{t}_k-\mathbf{t}_j)\right\}^{\alpha/2}\right] - \\
	&\cos \left(\mathbf{t}^T_j\boldsymbol\mu\right) \cdot \sin \left(\mathbf{t}^T_k\boldsymbol\mu \right) \cdot \exp\left[-\left\{\frac{1}{2}\mathbf{t}_j^T \boldsymbol\Sigma \mathbf{t}_j\right\}^{\alpha/2}\right] \cdot \exp\left[-\left\{\frac{1}{2}\mathbf{t}_k^T \boldsymbol\Sigma \mathbf{t}_k\right\}^{\alpha/2}\right] 
	\end{align*}
    \begin{align*}
	[\mathbf{\Omega}_{RI}]_{jk} =& \operatorname{Cov}\left[ \rm{Re}(e^{i\mathbf{t}^T_jX}), \rm{Im}(e^{i\mathbf{t}^T_kX}) \right] = \frac{1}{2} \rm{Im}(\phi(\mathbf{t}_j+\mathbf{t}_k)) + \frac{1}{2} \rm{Im}(\phi(-\mathbf{t}_j+\mathbf{t}_k)) -
	\rm{Im}[\phi(\mathbf{t}_j)] \cdot \rm{Re} \{\phi(\mathbf{t}_k)\} \\
	&\frac{1}{2}  \sin \left(\mathbf{t}^T_j\boldsymbol\mu+\mathbf{t}^T_k\boldsymbol\mu \right) \cdot \exp\left[- \left\{\frac{1}{2} (\mathbf{t}_j+\mathbf{t}_k)^T \boldsymbol\Sigma (\mathbf{t}_j+\mathbf{t}_k)\right\}^{\alpha/2}\right] + \\ 
	&\frac{1}{2}  \sin \left(\mathbf{t}^T_k\boldsymbol\mu-\mathbf{t}^T_j\boldsymbol\mu \right) \cdot \exp\left[-\left\{\frac{1}{2} (\mathbf{t}_k-\mathbf{t}_j)^T \boldsymbol\Sigma (\mathbf{t}_k-\mathbf{t}_j)\right\}^{\alpha/2}\right] - \\
	&\sin \left(\mathbf{t}^T_j\boldsymbol\mu\right) \cdot \cos \left(\mathbf{t}^T_k\boldsymbol\mu \right) \exp\left[-\left\{\frac{1}{2}\mathbf{t}_j^T \boldsymbol\Sigma \mathbf{t}_j\right\}^{\alpha/2}\right] \exp\left[-\left\{\frac{1}{2}\mathbf{t}_k^T \boldsymbol\Sigma \mathbf{t}_k\right\}^{\alpha/2}\right] \\ 
	\end{align*}
	$j,k = 1 \dots m $
\end{prop}

The proof is given in Section \ref{sec:Proofs}.

In the following theorem we apply the delta method to the statistic (\ref{eq:phi_empiric}) and CLT (\ref{feuer_CLT_parts_ecf}) to obtain the limiting distribution of $(\hat \alpha, \hat{\boldsymbol\Sigma}_d, \hat{\boldsymbol\Sigma}_{nd})$. 
\begin{thm}	
\label{thm:conf_interv}
Under Assumption \ref{ass:Sigma_d_n_0}
\begin{equation}
\label{CLT_main_statem}
    \sqrt{n}\begin{pmatrix}
    \hat\alpha^{mult} - \alpha \\
    \hat{\boldsymbol\Sigma}_d - \boldsymbol\Sigma_d \\
    \hat{\boldsymbol\Sigma}_{nd} - \boldsymbol\Sigma_{nd}    
    \end{pmatrix} 
    \xrightarrow[n \to \infty]{d}
    \mathbb{N}_{p(p+1)/2 + 1}(0, \mathbf{G}^\top \mathbf{\Omega} \mathbf{G})
\end{equation}
where $\mathbf{\Omega}$ is defined in (\ref{Omega_repres}) and
$\mathbf{G}$ is a $(\frac{p^2+p}{2}+1) \times 2(p^2+p)$ matrix whose elements are given below:
\begin{equation}
\mathbf{G} = \begin{pmatrix}
 \left(\boldsymbol{A}_s^1 \odot \boldsymbol{R}_{\phi}^1\right)^T & \left(\boldsymbol{A}_s^2 \odot \boldsymbol{R}_{\phi}^2\right)^T & \boldsymbol{0} & \boldsymbol{0} &  \left(\boldsymbol{A}_s^1 \odot \boldsymbol{Im}_{\phi}^1\right)^T & \left(\boldsymbol{A}_s^2 \odot \boldsymbol{Im}_{\phi}^2\right)^T & \boldsymbol{0} & \boldsymbol{0}\\
 \boldsymbol{B}^{r}_1 & \boldsymbol{B}^{r}_2 & \boldsymbol{0} & \boldsymbol{0} & \boldsymbol{B}^{i}_1 & \boldsymbol{B}^{i}_2 & \boldsymbol{0} & \boldsymbol{0}\\
 \boldsymbol{J}^r_1 & \boldsymbol{J}^r_2 & \boldsymbol{J}^r_+ & \boldsymbol{J}^r_- & \boldsymbol{J}^i_1 & \boldsymbol{J}^i_2 & \boldsymbol{J}^i_+ & \boldsymbol{J}^i_-\\ 
\end{pmatrix} 
\end{equation}
\begin{equation}
\begin{split}
        \boldsymbol{A}_s^1 [k]:=\left[\log \big|\frac{s_1}{s_2}\big| \cdot \log|\phi(s_1\mathbf{e}_{k})| \cdot |\phi(s_1 \mathbf{e}_k)| \right]^{-1} \qquad
        \boldsymbol{A}_s^2 [k]:=-\left[ \log \left|\frac{s_1}{s_2} \right| \cdot \log|\phi(s_2\mathbf{e}_{k})| \cdot |\phi(s_2 \mathbf{e}_k)| \right]^{-1}        
\end{split}
\end{equation}
\begin{equation}
\begin{split}
    \boldsymbol{B}^r_1[j,k] = \mathbf{b}_1[j,k] \cdot \frac{\rm{Re}[\phi(s_1 \mathbf{e}_k)]}{|\phi(s_1 \mathbf{e}_k)|} \quad
    \boldsymbol{B}^i_1[j,k] = \mathbf{b}_1[j,k] \cdot \frac{\rm{Im}[\phi(s_1 \mathbf{e}_k)]}{|\phi(s_1 \mathbf{e}_k)|} \quad
    \boldsymbol{b}_1[j,k] =    \frac{-\boldsymbol\Sigma_{jj} \cdot \log \left( \frac{s^2_1}{2} \boldsymbol\Sigma_{jj} \right)}{
    \log\left|\frac{s_1}{s_2}\right| \cdot   \log|\phi(s_1\mathbf{e}_k)| \cdot |\phi(s_1 \mathbf{e}_k)|}
\end{split}
\end{equation}
$\forall j=1, \dots, p; k=1, \dots, p \text{ such that } j \neq k.$
\begin{equation}
    diag(\mathbf{B}^i_1) = \mathbf{b}_d \odot \boldsymbol{Im}_{\phi}^1 \quad
    diag(\mathbf{B}^r_1) = \mathbf{b}_d \odot \boldsymbol{R}_{\phi}^1  \text{, where } \quad
    \mathbf{b}_d[j] = \frac{2\boldsymbol\Sigma_{jj}}{\alpha \cdot \log| \phi(s_1\mathbf{e}_j)| \cdot | \phi(s_1\mathbf{e}_j)|} \left(1 - \frac{\log\left(\frac{s_1^2 \boldsymbol\Sigma_{jj}}{2}\right)}{2\log\left|\frac{s_1}{s_2}\right|}\right)    
\end{equation}
and $\boldsymbol{R}_{\phi}^1$ and $\boldsymbol{Im}_{\phi}^1$ are defined by (\ref{def:ARIm_1}). Here, $\odot$ stands for Hadamard product of arrays.
\begin{equation}
\begin{split}
    \boldsymbol{B}^{r}_2[j,k] = \boldsymbol{B}_2[j,k] \cdot \frac{\rm{Re}[\phi(s_2 \mathbf{e}_k)]}{|\phi(s_2 \mathbf{e}_k)|} \quad \boldsymbol{B}^{i}_2[j,k] = \boldsymbol{B}_2[j,k] \cdot \frac{\rm{Im}[\phi(s_2 \mathbf{e}_k)]}{|\phi(s_2 \mathbf{e}_k)|} \\ 
    \boldsymbol{B}_2[j,k] = \frac{\boldsymbol\Sigma_{jj} \cdot \log \left( \frac{s^2_1}{2} \boldsymbol\Sigma_{jj} \right)}{
    \alpha \cdot \log\left|\frac{s_1}{s_2}\right| \cdot \log|\phi(s_2\mathbf{e}_k)| \cdot |\phi(s_2 \mathbf{e}_k)|}
\end{split}
\end{equation}
$\forall j=1, \dots, p; k=1, \dots, p$.
\begin{equation}
\begin{split}
    &\boldsymbol{J^r_1}\big(\#(i,j), k\big) = \frac{\mathbf{D}_1\big(\#(i,j)\big) \cdot \rm{Re}[\phi(s_1 \mathbf{e}_k)]}{\log|\phi(s_1\mathbf{e}_k)| \cdot |\phi(s_1\mathbf{e}_k)|^2}     
    \quad
    \boldsymbol{J^i_1}\big(\#(i,j), k\big) = \frac{\mathbf{D}_1\big(\#(i,j)\big) \cdot \rm{Im}[\phi(s_1 \mathbf{e}_k)]}{\log|\phi(s_1\mathbf{e}_k)| \cdot |\phi(s_1\mathbf{e}_k)|^2} 
    \\
    &\boldsymbol{J^r_2}\big(\#(i,j), k\big) = \frac{\mathbf{D}_1\big(\#(i,j)\big) \cdot \rm{Re}[\phi(s_2 \mathbf{e}_k)]}{- \log|\phi(s_2\mathbf{e}_k)| \cdot |\phi(s_2\mathbf{e}_k)|^2} 
    \quad
    \boldsymbol{J^i_2}\big(\#(i,j), k\big) = \frac{\mathbf{D}_1\big(\#(i,j)\big) \cdot \rm{Im}[\phi(s_2 \mathbf{e}_k)]}{- \log|\phi(s_2\mathbf{e}_k)| \cdot |\phi(s_2\mathbf{e}_k)|^2 }          
\end{split}
\end{equation}
\begin{equation}
\begin{split}
    &\mathbf{D}_1(\#(i,j)) = 
    \Big[ \Big(\frac{1}{2}\boldsymbol\Sigma_{ii} + \boldsymbol\Sigma_{ij} + \frac{1}{2}\boldsymbol\Sigma_{jj}\Big) \cdot  \log \Big(\frac{1}{2}\boldsymbol\Sigma_{ii} + \boldsymbol\Sigma_{ij} + \frac{1}{2}\boldsymbol\Sigma_{jj}\Big)- 
    \\
    &\Big(\frac{1}{2}\boldsymbol\Sigma_{ii} - \boldsymbol\Sigma_{ij} + \frac{1}{2}\boldsymbol\Sigma_{jj}\Big) \cdot  \log \Big(\frac{1}{2}\boldsymbol\Sigma_{ii} - \boldsymbol\Sigma_{ij} + \frac{1}{2}\boldsymbol\Sigma_{jj}\Big)\Big] \cdot \left(2 \alpha \cdot \log\left|\frac{s_1}{s_2}\right|\right)^{-1} \text{ is a column-vector.}
\end{split}
\end{equation}
$\#$ is defined in (\ref{eq:index_hash}), $i \in [1, p], j \in [1, i-1]; k \in [1, p]$.
\begin{equation}
\begin{split}
    &\boldsymbol{J}^r_+(k,k') = 0 \text{ and }
    \boldsymbol{J}^i_+(k,k') = 0 \text{ when  } k \neq k'.
    \\
    &diag(\boldsymbol{J}^r_+) = \boldsymbol{d}_+ \odot \boldsymbol{R}^+_{\phi} \quad 
    diag(\boldsymbol{J}^i_+) = \boldsymbol{d}_+ \odot \boldsymbol{Im}^+_{\phi} \text{, where} 
    \\
    &\boldsymbol{d}_+[\#(i,j)] = -\frac{1}{\alpha \cdot |\phi(\mathbf{e}_i+\mathbf{e}_j)|} \left(\frac{1}{2} \boldsymbol\Sigma_{ii} + \boldsymbol\Sigma_{ij} + \frac{1}{2} \boldsymbol\Sigma_{jj}\right)^{(1-\frac{\alpha}{2})}    
\end{split}
\end{equation}
\begin{equation}
\begin{split}
    \boldsymbol{R}^+_{\phi}[\#(i,j)] = \frac{ \operatorname{Re} \left(\phi(\mathbf{e}_i+\mathbf{e}_j)\right)}{|\phi(\mathbf{e}_i+\mathbf{e}_j)|}
    \quad 
    \boldsymbol{Im}^+_{\phi}[\#(i,j)] = \frac{\operatorname{Im} \left(\phi(\mathbf{e}_i+\mathbf{e}_j)\right)}{|\phi(\mathbf{e}_i+\mathbf{e}_j)|}
\end{split}
\end{equation}

\begin{equation}
\begin{split}
&\boldsymbol{J}^r_-(k,k') = 0 \text{ and }
\boldsymbol{J}^i_-(k,k') = 0 \text{ when  } k \neq k'.
\\
&diag(\boldsymbol{J}^r_-) = \boldsymbol{d}_- \odot \boldsymbol{R}^-_{\phi} \quad 
diag(\boldsymbol{J}^i_-) = \boldsymbol{d}_- \odot \boldsymbol{Im}^-_{\phi} \text{, where} \\
&\boldsymbol{d}_-[\#(i,j)] = \frac{1}{\alpha \cdot |\phi(\mathbf{e}_i-\mathbf{e}_j)|} \left(\frac{1}{2} \boldsymbol\Sigma_{ii} - \boldsymbol\Sigma_{ij} + \frac{1}{2} \boldsymbol\Sigma_{jj}\right)^{(1-\frac{\alpha}{2})}    
\end{split}
\end{equation}
\begin{equation}
\begin{split}
\boldsymbol{R}^+_{\phi}[\#(i,j)] = \frac{ \operatorname{Re} \left(\phi(\mathbf{e}_i-\mathbf{e}_j)\right)}{|\phi(\mathbf{e}_i-\mathbf{e}_j)|}
\quad 
\boldsymbol{Im}^+_{\phi}[\#(i,j)] = \frac{\operatorname{Im} \left(\phi(\mathbf{e}_i-\mathbf{e}_j)\right)}{|\phi(\mathbf{e}_i-\mathbf{e}_j)|}
\end{split}
\end{equation}
\end{thm}

The proof is given in Section \ref{sec:Proofs}.

The form of the covariance matrix of the limiting distribution in (\ref{CLT_main_statem}) has an advantage and a disadvantage. On one hand, it allows to extend Theorem \ref{CLT_main_statem} to include $\hat{\boldsymbol\mu}[k]$ estimator from (\ref{def:mu_estimator}). In such a case Jacobian matrix $\mathbf{G}$ would be augmented with entries corresponding to $\hat{\boldsymbol\mu}[k]$. On the other hand, the covariance matrix is cumbersome and hard to work with in practice, but luckily could be simplified. In particular, the number of entries in matrix $\mathbf{G}$ can be reduced two times. The following corollary trades of the extensibility for simplicity.
\begin{cor}
	Let $\mathbf{x}\sim S_p(\boldsymbol\mu, \boldsymbol\Sigma, \boldsymbol{\alpha})$. Consider vector $w: w_k=\exp(i\tilde{\mathbf{t}}^\top_k \mathbf{x})$ with $\tilde{\mathbf{t}}$ defined in (\ref{def:t_tilde}). Under Assumption \ref{ass:Sigma_d_n_0}
	\begin{equation}
	\sqrt{n}\begin{pmatrix}
	\hat\alpha^{mult} - \alpha \\
	\hat{\boldsymbol\Sigma}_d - \boldsymbol\Sigma_d \\
	\hat{\boldsymbol\Sigma}_{nd} - \boldsymbol\Sigma_{nd}    
	\end{pmatrix} 
	\xrightarrow[n \to \infty]{d}
	\mathbb{N}_{p(p+1)/2 + 1}(0, \tilde{\mathbf{G}}^\top \tilde{\mathbf{\Omega}} \tilde{\mathbf{G}}) \quad,
	\end{equation}
	where $\tilde{\mathbf{\Omega}}$ is the covariance matrix of $w$ and
	$\tilde{\mathbf{G}}$ is a $(\frac{p^2+p}{2}+1) \times (p^2+p)$ matrix whose elements are given below.
	\begin{equation}
	\mathbf{G} = \begin{pmatrix}
	\left(\boldsymbol{A}_s^1 \right)^T & \left(\boldsymbol{A}_s^2 \right)^T & \boldsymbol{0} & \boldsymbol{0}\\
	\boldsymbol{B}_1 & \boldsymbol{B}_2 & \boldsymbol{0} & \boldsymbol{0}\\
	\boldsymbol{J}_1 & \boldsymbol{J}_2 & \boldsymbol{J}_+ & \boldsymbol{J}_- \\ 
	\end{pmatrix} 
	\end{equation}
	\begin{equation}
	\begin{split}
	\boldsymbol{B}_1[j,k] = \boldsymbol{b}_1[j,k] \quad \forall j=1, \dots, p; k=1, \dots, p \text{ such that } j \neq k.
	\end{split}
	\end{equation}
	
	\begin{equation}
	diag(\boldsymbol{B}_1) = \mathbf{b}_d
	\end{equation}
	\begin{equation}
	\begin{split}
	&\boldsymbol{J_1}\big(\#(i,j), k\big) = \frac{\mathbf{D}_1\big(\#(i,j)\big)}{\log|\phi(s_1\mathbf{e}_k)| \cdot |\phi(s_1\mathbf{e}_k)|}     
	\quad
	\boldsymbol{J_2}\big(\#(i,j), k\big) = \frac{\mathbf{D}_1\big(\#(i,j)\big)}{- \log|\phi(s_2\mathbf{e}_k)| \cdot |\phi(s_2\mathbf{e}_k)| }          
	\end{split}
	\end{equation}
	\begin{equation}
	\begin{split}
	&\boldsymbol{J}_+(k,k') = 0 \text{ when  } k \neq k' \text{ and } 
	diag(\boldsymbol{J}_+) = \boldsymbol{d}_+    
	\\
	&\boldsymbol{J}_-(k,k') = 0 \text{ when  } k \neq k' \text{ and } diag(\boldsymbol{J}_-) = \boldsymbol{d}_- \quad   
	\end{split}
	\end{equation}
\end{cor}
\begin{proof}
	The corollary follows from the proof of Theorem (\ref{thm:conf_interv}).
\end{proof}
We do not investigate the exact form of $\tilde{\mathbf{\Omega}}$ and leave it for future research.

\section{Numerical study}
This section consists of Monte-Carlo experiments showing performance of the estimators discussed in the previous sections. The setup of every experiment is the same: 2000 samples of 100, 1000 or 10 000 3-dimensional iid S$\alpha$S variables are generated for each alpha from the set $\{0.5, 1.0, 1.5\}$.

\subsection{Performance of different alpha-estimators}
In this part, numerical performance of 3 estimators is explored: the original estimator by S.J. Press (\ref{eqn:alpha_hat}), the estimator based on only 1 component of the sub-Gaussian vector (\ref{def:alpha_s}) and the last one (\ref{def:alpha_m}).

In the following experiments we illustrate issues with alpha estimators. As we mentioned earlier, the problem of $\hat \alpha^p$ is that it basically aggregates all components of $\mathbf{x}$ into one statistic with potentially large deviations. Analogously, $\hat \alpha_s$ exploits only 1 component of the vector which might have a wide distribution.

We also demonstrate that $\hat \alpha^{mult}$ shows better performance and explain it by its internal structure. 

\paragraph{Covariance matrix with 3 equally dispersed independent components}

\begin{align}
\boldsymbol\Sigma = \begin{pmatrix} 
0.1 & 0 & 0 \\
0 & 0.1 & 0 \\
0 & 0 & 0.1 \\
\end{pmatrix}
\label{Num:Sigma_diagonal}
\end{align}
\begin{table}[H]
	\centering
	\begin{tabular}{|r|r|r|r|r|r|r|r|}
		\hline
		sample\_size & alpha & alpha\_p\_b & alpha\_s\_b & alpha\_m\_b & alpha\_p\_rm & alpha\_s\_rm & alpha\_m\_rm\\
		\hline
		100 & 0.5 & -0.0011 & 0.0032 & 0.0005 & 0.2532 & 0.2322 & 0.1422\\
		
		1000 & 0.5 & -0.0036 & -0.0009 & 0.0010 & 0.0791 & 0.0718 & 0.0453\\
		
		10000 & 0.5 & -0.0003 & 0.0000 & -0.0001 & 0.0248 & 0.0227 & 0.0141\\
		\hline
		100 & 1.0 & -0.0380 & 0.0018 & 0.0037 & 0.2769 & 0.2255 & 0.1447\\
		
		1000 & 1.0 & -0.0032 & -0.0013 & -0.0002 & 0.0915 & 0.0698 & 0.0452\\
		
		10000 & 1.0 & -0.0002 & -0.0011 & -0.0006 & 0.0280 & 0.0220 & 0.0142\\
		\hline
		100 & 1.5 & -0.1470 & 0.0062 & 0.0024 & 0.3226 & 0.1960 & 0.1278\\
		
		1000 & 1.5 & -0.0087 & 0.0021 & 0.0004 & 0.1310 & 0.0624 & 0.0415\\
		
		10000 & 1.5 & -0.0007 & 0.0001 & 0.0000 & 0.0413 & 0.0201 & 0.0130\\
		\hline
	\end{tabular}
	\caption{Biases and RMSE's of $\hat \alpha^p$, $\hat \alpha_s$ and $\hat \alpha^{mult}$ for different values of alpha and the sample size. $\boldsymbol\Sigma$ is the one in (\ref{Num:Sigma_diagonal}).}
\end{table}

\paragraph{Covariance matrix with 3 equally dispersed inter-dependent components}

\begin{align}
\boldsymbol\Sigma = \begin{pmatrix} 
0.10 & 0.04 & 0.01 \\
0.04 & 0.10 & 0.02 \\
0.01 & 0.02 & 0.10 \\
\end{pmatrix}
\label{Num:Sigma_common}
\end{align}
\begin{table}[H]
	\centering
	\begin{tabular}{|r|r|r|r|r|r|r|r|}
		\hline
		sample\_size & alpha & alpha\_p\_b & alpha\_s\_b & alpha\_m\_b & alpha\_p\_rm & alpha\_s\_rm & alpha\_m\_rm\\
		\hline
		100 & 0.5 & -0.0038 & 0.0027 & -0.0012 & 0.2633 & 0.2395 & 0.1453\\
		
		1000 & 0.5 & 0.0007 & 0.0007 & -0.0006 & 0.0821 & 0.0716 & 0.0443\\
		
		10000 & 0.5 & 0.0005 & 0.0009 & 0.0003 & 0.0261 & 0.0227 & 0.0139\\
		\hline
		100 & 1.0 & -0.0684 & 0.0086 & 0.0023 & 0.3079 & 0.2255 & 0.1436\\
		
		1000 & 1.0 & -0.0081 & -0.0001 & 0.0009 & 0.1098 & 0.0693 & 0.0449\\
		
		10000 & 1.0 & -0.0021 & 0.0003 & 0.0001 & 0.0358 & 0.0224 & 0.0143\\
		\hline
		100 & 1.5 & -0.3924 & -0.0023 & -0.0006 & 0.4981 & 0.1979 & 0.1305\\
		
		1000 & 1.5 & -0.0655 & 0.0014 & 0.0009 & 0.1940 & 0.0628 & 0.0418\\
		
		10000 & 1.5 & -0.0033 & 0.0006 & 0.0005 & 0.0797 & 0.0192 & 0.0127\\
		\hline
	\end{tabular}
	\caption{Biases and RMSE's of $\hat \alpha^p$, $\hat \alpha_s$ and $\hat \alpha^{mult}$ for different values of alpha and the sample size. $\boldsymbol\Sigma$ is the one in (\ref{Num:Sigma_common}).}
	\label{Num:Results_common}
\end{table}

In these two experiments $\hat \alpha^p$ shows the worst performance. Performance of $\hat \alpha_s$ is between those of $\hat \alpha^p$ and $\hat \alpha^{mult}$ even though all variances in $\Sigma$ are the same. $\hat \alpha^{mult}$ is the most precise estimator probably due the fact that $\hat \alpha^{mult}$ uses more data of "the same variation".

\paragraph{Covariance matrix with 1 outstandingly dispersed component}
\begin{align}
\label{Num:Sigma_common_outsand_el}
\boldsymbol\Sigma = \begin{pmatrix} 
1.0 & 0.04 & 0.01 \\
0.04 & 0.10 & 0.02 \\
0.01 & 0.02 & 0.10 \\
\end{pmatrix}
\end{align}

\begin{table}[H]
	\centering
	\begin{tabular}{|r|r|r|r|r|r|r|r|}
		\hline
		sample\_size & alpha & alpha\_p\_b & alpha\_s\_b & alpha\_m\_b & alpha\_p\_rm & alpha\_s\_rm & alpha\_m\_rm\\
		\hline
		100 & 0.5 & -0.0321 & -0.0225 & -0.0046 & 0.3117 & 0.2991 & 0.1542\\
		
		1000 & 0.5 & 0.0045 & 0.0003 & 0.0003 & 0.1035 & 0.0955 & 0.0464\\
		
		10000 & 0.5 & 0.0000 & 0.0012 & -0.0001 & 0.0327 & 0.0310 & 0.0158\\
		\hline
		100 & 1.0 & -0.5070 & -0.3510 & -0.1150 & 0.6103 & 0.4847 & 0.1962\\
		
		1000 & 1.0 & -0.1591 & -0.0624 & -0.0213 & 0.2493 & 0.1929 & 0.0738\\
		
		10000 & 1.0 & -0.0067 & -0.0017 & -0.0005 & 0.1122 & 0.0759 & 0.0274\\
		\hline
		100 & 1.5 & -1.2390 & -1.0204 & -0.3375 & 1.2925 & 1.0752 & 0.3712\\
		
		1000 & 1.5 & -0.8832 & -0.6333 & -0.2120 & 0.9058 & 0.6612 & 0.2235\\
		
		10000 & 1.5 & -0.5792 & -0.3508 & -0.1170 & 0.5965 & 0.3762 & 0.1258\\
		\hline
	\end{tabular}
	\caption{Biases and RMSE's of $\hat \alpha^p$, $\hat \alpha_s$ and $\hat \alpha^{mult}$ for different values of alpha and the sample size. $\boldsymbol\Sigma$ is the one in (\ref{Num:Sigma_common_outsand_el})}
\end{table}
Comparing these results with the ones in Table \ref{Num:Results_common}, where there is no extremely large component, it is seen that both $\hat\alpha^p$ and $\hat\alpha_s$ got increases in their bias and root mean square error. Performance of $\hat\alpha^p$ here is comparable to that of $\hat\alpha_s$ because of the dominant component $\boldsymbol\Sigma[1,1]$ influencing the empirical characteristic function in (\ref{eqn:alpha_hat}). At the same time,  performance of $\hat\alpha^{mult}$ decreased less, which suggests that $\hat\alpha^{mult}$ is more robust to highly influential components of $\boldsymbol\Sigma$ due to component-wise estimation of $\alpha$.

\subsubsection{Estimation of $\alpha$ when elements of $\Sigma$ are not within [0.1 -- 1].}

\paragraph{Covariance matrix with more dispersed inter-dependent components}
\begin{align}
\boldsymbol\Sigma = \begin{pmatrix} 
10 & 4 & 1 \\
4 & 10 & 2 \\
1 & 2 & 10 \\
\end{pmatrix}
\label{Num:Sigma_large}
\end{align}
\begin{table}[H]
	\centering
	\begin{tabular}{|r|r|r|r|r|r|r|r|}
		\hline
		sample\_size & alpha & alpha\_p\_b & alpha\_s\_b & alpha\_m\_b & alpha\_p\_rm & alpha\_s\_rm & alpha\_m\_rm\\
		\hline
		100 & 0.5 & -0.445 & -0.280 & -0.280 & 0.579 & 0.471 & 0.358\\
		
		1000 & 0.5 & -0.279 & -0.036 & -0.036 & 0.370 & 0.195 & 0.117\\
		
		10000 & 0.5 & -0.063 & 0.002 & -0.001 & 0.155 & 0.073 & 0.042\\
		\hline
		100 & 1.0 & -0.991 & -0.995 & -0.997 & 1.055 & 1.061 & 1.018\\
		
		1000 & 1.0 & -1.007 & -0.978 & -0.979 & 1.039 & 1.009 & 0.990\\
		
		10000 & 1.0 & -0.998 & -0.890 & -0.891 & 1.016 & 0.909 & 0.898\\
		\hline
		100 & 1.5 & -1.504 & -1.509 & -1.504 & 1.547 & 1.551 & 1.519\\
		
		1000 & 1.5 & -1.499 & -1.499 & -1.497 & 1.520 & 1.520 & 1.505\\
		
		10000 & 1.5 & -1.499 & -1.506 & -1.500 & 1.513 & 1.518 & 1.504\\
		\hline
	\end{tabular}
	\caption{Biases and RMSE's of $\hat \alpha^p$, $\hat \alpha_s$ and $\hat \alpha^{mult}$ for different values of alpha and the sample size. $\boldsymbol\Sigma$ is the one in (\ref{Num:Sigma_large})}
\end{table}
As in other stable models, deviations from 0.1-1 by order of magnitude in the scaling parameter tend to decrease accuracy of estimation. All 3 $\alpha$ estimators lost in quality of performance. In the following table the estimators perform reasonably well only in the case $\alpha=0.5$. In cases of other two alphas the errors are roughly 100\% of the parameter values.

\paragraph{Covariance matrix with less dispersed inter-dependent components}
\begin{align}
\boldsymbol\Sigma = \begin{pmatrix} 
1 & 0.4 & 0.1 \\
0.4 & 1 & 0.2 \\
0.1 & 0.2 & 1 \\
\end{pmatrix} \cdot 10^{-3}
\label{Num:Sigma_small}
\end{align}
\begin{table}[H]
	\centering
	\begin{tabular}{|r|r|r|r|r|r|r|r|}
		\hline
		sample\_size & alpha & alpha\_p\_b & alpha\_s\_b & alpha\_m\_b & alpha\_p\_rm & alpha\_s\_rm & alpha\_m\_rm\\
		\hline
		100 & 0.5 & 0.011 & 0.022 & 0.015 & 0.243 & 0.271 & 0.174\\
		
		1000 & 0.5 & 0.004 & 0.002 & 0.003 & 0.075 & 0.082 & 0.052\\
		
		10000 & 0.5 & 0.001 & 0.001 & 0.000 & 0.024 & 0.026 & 0.017\\
		\hline
		100 & 1.0 & 0.038 & 0.091 & 0.086 & 0.315 & 0.444 & 0.340\\
		
		1000 & 1.0 & 0.004 & 0.008 & 0.009 & 0.097 & 0.133 & 0.098\\
		
		10000 & 1.0 & 0.001 & -0.002 & 0.000 & 0.030 & 0.042 & 0.031\\
		\hline
		100 & 1.5 & 0.073 & 0.198 & 0.184 & 0.369 & 0.502 & 0.416\\
		
		1000 & 1.5 & 0.009 & 0.031 & 0.028 & 0.123 & 0.216 & 0.164\\
		
		10000 & 1.5 & 0.001 & 0.002 & 0.003 & 0.040 & 0.068 & 0.052\\
		\hline
	\end{tabular}
	\caption{Biases and RMSE's of $\hat \alpha^p$, $\hat \alpha_s$ and $\hat \alpha^{mult}$ for different values of alpha and the sample size. $\boldsymbol\Sigma$ is the one in (\ref{Num:Sigma_small})}
\end{table}

Expectedly, in this case $\hat\alpha_s$, and $\hat\alpha^{mult}$ also lost in accuracy, though this is not the case for $\hat\alpha^p$. The reason why $\hat\alpha^p$ improved is not clear.

In all the above experiments $\hat\alpha^{mult}$ is the best estimator out of the three.

\subsection{Sigma estimators}
To demonstrate performance of the diagonal and the non-diagonal Sigma estimators we use the same $\boldsymbol\Sigma$ matrix as in (\ref{Num:Sigma_common}):
\begin{align}
\boldsymbol\Sigma = \begin{pmatrix} 
0.10 & 0.04 & 0.01 \\
0.04 & 0.10 & 0.02 \\
0.01 & 0.02 & 0.10 \\
\end{pmatrix}
\label{Num:Sigma_common_rep}
\end{align}
\begin{table}[h!]
	\centering
	\begin{tabular}{|r|r|r|r|r|r|r|r|}
		\hline
		sample\_size & alpha & Sigma11\_b & Sigma22\_b & Sigma33\_b & Sigma11\_rm & Sigma22\_rm & Sigma33\_rm\\
		\hline
		100 & 0.5 & 0.0213 & 0.0291 & 0.0292 & 0.1156 & 0.1761 & 0.1868\\
		
		1000 & 0.5 & 0.0011 & 0.0020 & 0.0022 & 0.0234 & 0.0236 & 0.0235\\
		
		10000 & 0.5 & 0.0001 & 0.0000 & 0.0000 & 0.0072 & 0.0073 & 0.0069\\
		\hline
		100 & 1.0 & 0.0024 & 0.0019 & 0.0027 & 0.0367 & 0.0357 & 0.0368\\
		
		1000 & 1.0 & 0.0003 & 0.0002 & 0.0001 & 0.0114 & 0.0111 & 0.0108\\
		
		10000 & 1.0 & -0.0001 & 0.0000 & -0.0001 & 0.0035 & 0.0034 & 0.0034\\
		\hline
		100 & 1.5 & -0.0003 & -0.0004 & -0.0008 & 0.0220 & 0.0218 & 0.0229\\
		
		1000 & 1.5 & -0.0002 & -0.0001 & -0.0002 & 0.0067 & 0.0068 & 0.0069\\
		
		10000 & 1.5 & 0.0000 & 0.0000 & 0.0000 & 0.0022 & 0.0022 & 0.0021\\
		\hline
	\end{tabular}
	\caption{Biases and RMSE's of $\hat{\boldsymbol\Sigma}_d$ for different values of alpha and the sample size. $\boldsymbol\Sigma$ is the one in (\ref{Num:Sigma_common_rep})}
\end{table}
\begin{table}[h!]
	\centering
	\begin{tabular}{|r|r|r|r|r|r|r|r|}
		\hline
		sample\_size & alpha & Sigma21\_b & Sigma31\_b & Sigma32\_b & Sigma21\_rm & Sigma31\_rm & Sigma32\_rm\\
		\hline
		100 & 0.5 & 6e-03 & 0.0019 & 0.0032 & 0.0670 & 0.0610 & 0.0612\\
		
		1000 & 0.5 & 7e-04 & 0.0000 & 0.0002 & 0.0162 & 0.0136 & 0.0141\\
		
		10000 & 0.5 & 1e-04 & 0.0003 & 0.0002 & 0.0049 & 0.0042 & 0.0042\\
		\hline
		100 & 1.0 & 2e-04 & -0.0003 & 0.0005 & 0.0251 & 0.0232 & 0.0233\\
		
		1000 & 1.0 & 0e+00 & 0.0001 & 0.0000 & 0.0078 & 0.0070 & 0.0072\\
		
		10000 & 1.0 & 0e+00 & 0.0000 & 0.0001 & 0.0024 & 0.0022 & 0.0022\\
		\hline
		100 & 1.5 & -2e-04 & 0.0003 & 0.0001 & 0.0165 & 0.0160 & 0.0158\\
		
		1000 & 1.5 & 1e-04 & 0.0002 & 0.0001 & 0.0054 & 0.0051 & 0.0050\\
		
		10000 & 1.5 & 0e+00 & 0.0000 & 0.0000 & 0.0017 & 0.0016 & 0.0016\\
		\hline
	\end{tabular}
	\caption{Biases and RMSE's of $\hat{\boldsymbol\Sigma}_{nd}$ for different values of alpha and the sample size. $\boldsymbol\Sigma$ is the one in (\ref{Num:Sigma_common_rep})}
\end{table}

We observe that both estimators are unbiased. Their precision becomes reasonably good when the sample size is greater than 1000. Note that although $\boldsymbol\Sigma_{21}$ is 4 times greater than $\boldsymbol\Sigma_{31}$, the precisions of their estimates are very close to each other.

\section{Proofs}
\label{sec:Proofs}
\begin{proof}[Proof of Theorem \ref{CLT_Re_Im_vect}]
	(\ref{feuer_CLT_parts_ecf}) follows directly from the classical CLT, so the rest of the proof is devoted to obtaining $\mathbf{\Omega}$. 
    Let's temporally use the following notation: $Y(t):=e^{i\mathbf{t}^T\mathbf{X}}$, and establish a useful identity.
	\begin{equation}
	\begin{split}
		\mbox{Cov} \big[ Y(\mathbf{t}), Y(\boldsymbol{\tau}) \big] = \mbox{Cov} \big[ e^{i\mathbf{t}^T\mathbf{X}}, e^{i\boldsymbol\tau^T \mathbf{X}} \big] = \\
		\mathbb{E} \left( e^{i\mathbf{t}^T\mathbf{X}} \cdot  e^{i\boldsymbol\tau^T \mathbf{X}} \right) - \mathbb{E} \left( e^{i\mathbf{t}^T\mathbf{X}}\right) \cdot \mathbb{E} \left( e^{i\boldsymbol\tau^T \mathbf{X}}\right) = \phi(\mathbf{t}+\boldsymbol\tau) - \phi(\mathbf{t}) \cdot \phi(\boldsymbol\tau)
	\end{split}		
	\end{equation}	
	Then, using $\mbox{Cov} \big[ Y(\mathbf{t}), Y(\boldsymbol\tau) \big]$, we calculate all the elements of $\mathbf{\Omega}$:
    \begin{equation}
    \begin{split}
        \mbox{Cov} \left[ \operatorname{Re}(e^{i\mathbf{t}^T_j\mathbf{X}}), \operatorname{Im}(e^{i\mathbf{t}^T_l\mathbf{X}}) \right] = \mbox{Cov} \left( \frac{1}{2}\big[Y(\mathbf{t}_j) + Y(-\mathbf{t}_j)\big], \frac{1}{2}\big[Y(\mathbf{t}_l) - Y(-\mathbf{t}_l)\big] \right)= \\
        \frac{1}{4} \mbox{Cov} \big[Y(\mathbf{t}_j), Y(\mathbf{t}_l)\big] + \frac{1}{4} \mbox{Cov} \big[Y(-\mathbf{t}_j), Y(\mathbf{t}_l)\big] -
        \frac{1}{4} \mbox{Cov} \big[Y(\mathbf{t}_j), Y(-\mathbf{t}_l)\big] -
        \frac{1}{4} \mbox{Cov} \big[Y(-\mathbf{t}_j), Y(-\mathbf{t}_l)\big] = \\
        \frac{1}{4} \big[\phi(\mathbf{t}_j+\mathbf{t}_l) - \phi(\mathbf{t}_j) \phi(\mathbf{t}_l)\big] +
        \frac{1}{4} \big[\phi(-\mathbf{t}_j+\mathbf{t}_l) - \phi(-\mathbf{t}_j) \phi(\mathbf{t}_l) \big] - \\
        \frac{1}{4} \big[\phi(\mathbf{t}_j-\mathbf{t}_l) - \phi(\mathbf{t}_j) \phi(-\mathbf{t}_l)\big] -
        \frac{1}{4} \big[\phi(-\mathbf{t}_j-\mathbf{t}_l) - \phi(-\mathbf{t}_j) \phi(-\mathbf{t}_l)\big] = \\
        \frac{1}{2} \cdot \operatorname{Im}(\phi(\mathbf{t}_j+\mathbf{t}_l)) + \frac{1}{2} \cdot \operatorname{Im}(\phi(-\mathbf{t}_j+\mathbf{t}_l)) -
        \frac{1}{2} \operatorname{Im}[\phi(\mathbf{t}_j)\phi(\mathbf{t}_l)] +
        \frac{1}{2} \operatorname{Im}[\phi(\mathbf{t}_j)\phi(-\mathbf{t}_l)] = \\
        \frac{1}{2} \cdot \operatorname{Im}(\phi(\mathbf{t}_j+\mathbf{t}_l)) + \frac{1}{2} \cdot \operatorname{Im}(\phi(-\mathbf{t}_j+\mathbf{t}_l)) -
        \frac{1}{2} \operatorname{Im}\big[\phi(\mathbf{t}_j)\phi(\mathbf{t}_l) + \phi(\mathbf{t}_j)\phi(-\mathbf{t}_l)\big] = \\
        \frac{1}{2} \cdot \operatorname{Im}(\phi(\mathbf{t}_j+\mathbf{t}_l)) + \frac{1}{2} \cdot \operatorname{Im}(\phi(-\mathbf{t}_j+\mathbf{t}_l)) -
        \frac{1}{2} \operatorname{Im}\big[\phi(\mathbf{t}_j) \cdot 2\operatorname{Re} \{\phi(\mathbf{t}_l)\}\big] = \\
        \frac{1}{2} \cdot \operatorname{Im}(\phi(\mathbf{t}_j+\mathbf{t}_l)) + \frac{1}{2} \cdot \operatorname{Im}(\phi(-\mathbf{t}_j+\mathbf{t}_l)) -
        \operatorname{Im}[\phi(\mathbf{t}_j)] \cdot \operatorname{Re} \{\phi(\mathbf{t}_l)\}
    \end{split}		
    \end{equation}  

    \begin{equation}
    \begin{split}
        \operatorname{Cov} \left[ \operatorname{Re}(e^{i\mathbf{t}^T_j\mathbf{X}}), \operatorname{Re}(e^{i\mathbf{t}^T_l\mathbf{X}}) \right] = \operatorname{Cov} \left( \frac{1}{2}\big[Y(\mathbf{t}_j) + Y(-\mathbf{t}_j)\big], \frac{1}{2}\big[Y(\mathbf{t}_l) + Y(-\mathbf{t}_l)\big] \right)= \\
        \frac{1}{4} \operatorname{Cov} \big[Y(\mathbf{t}_j), Y(\mathbf{t}_l)\big] + \frac{1}{4} \operatorname{Cov} \big[Y(-\mathbf{t}_j), Y(\mathbf{t}_l)\big] +
        \frac{1}{4} \operatorname{Cov} \big[Y(\mathbf{t}_j), Y(-\mathbf{t}_l)\big] +
        \frac{1}{4} \operatorname{Cov} \big[Y(-\mathbf{t}_j), Y(-\mathbf{t}_l)\big] = \\
        \frac{1}{4} \left[\phi(\mathbf{t}_j+\mathbf{t}_l) - \phi(\mathbf{t}_j) \phi(\mathbf{t}_l)\right] +
        \frac{1}{4} \left[\phi(-\mathbf{t}_j+\mathbf{t}_l) - \phi(-\mathbf{t}_j) \phi(\mathbf{t}_l) \right] + \\
        \frac{1}{4} \left[\phi(\mathbf{t}_j-\mathbf{t}_l) - \phi(\mathbf{t}_j) \phi(-\mathbf{t}_l)\right] +
        \frac{1}{4} \left[\phi(-\mathbf{t}_j-\mathbf{t}_l) - \phi(-\mathbf{t}_j) \phi(-\mathbf{t}_l)\right] = \\
        \frac{1}{2} \cdot \operatorname{Re}(\phi(\mathbf{t}_j+\mathbf{t}_l)) + \frac{1}{2} \cdot \operatorname{Re}(\phi(-\mathbf{t}_j+\mathbf{t}_l)) -
        \frac{1}{2} \operatorname{Re}[\phi(\mathbf{t}_j)\phi(\mathbf{t}_l)] -
        \frac{1}{2} \operatorname{Re}[\phi(-\mathbf{t}_j)\phi(\mathbf{t}_l)] = \\
        \frac{1}{2} \cdot \operatorname{Re}(\phi(\mathbf{t}_j+\mathbf{t}_l)) + \frac{1}{2} \cdot \operatorname{Re}(\phi(-\mathbf{t}_j+\mathbf{t}_l)) -
        \frac{1}{2} \operatorname{Re}[\phi(\mathbf{t}_j)\phi(\mathbf{t}_l) + \phi(-\mathbf{t}_j)\phi(\mathbf{t}_l)] = \\
        \frac{1}{2} \cdot \operatorname{Re}(\phi(\mathbf{t}_j+\mathbf{t}_l)) + \frac{1}{2} \cdot \operatorname{Re}(\phi(-\mathbf{t}_j+\mathbf{t}_l)) -
        \frac{1}{2} \operatorname{Re}[\phi(\mathbf{t}_l)\cdot 2 \operatorname{Re} \{\phi(\mathbf{t}_j)\}] =  \\
        \frac{1}{2} \cdot \operatorname{Re}(\phi(\mathbf{t}_j+\mathbf{t}_l)) + \frac{1}{2} \cdot \operatorname{Re}(\phi(-\mathbf{t}_j+\mathbf{t}_l)) -
        \operatorname{Re}[\phi(\mathbf{t}_l)] \cdot  \operatorname{Re} \{\phi(\mathbf{t}_j)\}    
    \end{split}		
    \end{equation}    
    
    \begin{equation}
    \begin{split}
        \operatorname{Cov} \left[ \operatorname{Im}(e^{i\mathbf{t}^T_j\mathbf{X}}), \operatorname{Im}(e^{i\mathbf{t}^T_l\mathbf{X}}) \right] = \operatorname{Cov} \left( \frac{1}{2}\big[Y(\mathbf{t}_j) - Y(-\mathbf{t}_j)\big], \frac{1}{2}\big[Y(\mathbf{t}_l) - Y(-\mathbf{t}_l)\big] \right)= \\
        \frac{1}{4} \operatorname{Cov} \big[Y(\mathbf{t}_j), Y(\mathbf{t}_l)\big] - \frac{1}{4} \operatorname{Cov} \big[Y(-\mathbf{t}_j), Y(\mathbf{t}_l)\big] -
        \frac{1}{4} \operatorname{Cov} \big[Y(\mathbf{t}_j), Y(-\mathbf{t}_l)\big] +
        \frac{1}{4} \operatorname{Cov} \big[Y(-\mathbf{t}_j), Y(-\mathbf{t}_l)\big] = \\
        \frac{1}{4} \left[\phi(\mathbf{t}_j+\mathbf{t}_l) - \phi(\mathbf{t}_j) \phi(\mathbf{t}_l)\right] -
        \frac{1}{4} \left[\phi(-\mathbf{t}_j+\mathbf{t}_l) - \phi(-\mathbf{t}_j) \phi(\mathbf{t}_l) \right] - \\
        \frac{1}{4} \left[\phi(\mathbf{t}_j-\mathbf{t}_l) - \phi(\mathbf{t}_j) \phi(-\mathbf{t}_l)\right] +
        \frac{1}{4} \left[\phi(-\mathbf{t}_j-\mathbf{t}_l) - \phi(-\mathbf{t}_j) \phi(-\mathbf{t}_l)\right] = \\
        \frac{1}{2} \cdot \operatorname{Re}(\phi(\mathbf{t}_j+\mathbf{t}_l)) - \frac{1}{2} \cdot \operatorname{Re}(\phi(-\mathbf{t}_j+\mathbf{t}_l)) -
        \frac{1}{2} \operatorname{Re}[\phi(\mathbf{t}_j)\phi(\mathbf{t}_l)] +
        \frac{1}{2} \operatorname{Re}[\phi(\mathbf{t}_j)\phi(-\mathbf{t}_l)] = \\
        \frac{1}{2} \cdot \operatorname{Re}(\phi(\mathbf{t}_j+\mathbf{t}_l)) - \frac{1}{2} \cdot \operatorname{Re}(\phi(-\mathbf{t}_j+\mathbf{t}_l)) -
        \frac{1}{2} \operatorname{Re}\big[\phi(\mathbf{t}_j)\phi(\mathbf{t}_l) - \phi(\mathbf{t}_j)\phi(-\mathbf{t}_l)\big] = \\
        \frac{1}{2} \cdot \operatorname{Re}(\phi(\mathbf{t}_j+\mathbf{t}_l)) - \frac{1}{2} \cdot \operatorname{Re}(\phi(-\mathbf{t}_j+\mathbf{t}_l)) -
        \frac{1}{2} \operatorname{Re}\big[\phi(\mathbf{t}_j) \cdot 2 \operatorname{Im} \{\phi(\mathbf{t}_l)\}\big] = \\
        \frac{1}{2} \cdot \operatorname{Re}(\phi(\mathbf{t}_j+\mathbf{t}_l)) - \frac{1}{2} \cdot \operatorname{Re}(\phi(-\mathbf{t}_j+\mathbf{t}_l)) -
        \operatorname{Re}[\phi(\mathbf{t}_j)]  \cdot  \operatorname{Im} \{\phi(\mathbf{t}_l)\}     
    \end{split}		
    \end{equation}        
\end{proof}

\begin{proof}[Proof of Theorem \ref{thm:conf_interv}]		
	\item 
	\paragraph{The multidimensional delta method.}
	We start by defining four matrices:
	\begin{align}
	\begin{split}
	\mathbf{t}^e_j := (s_j \mathbf{e}_1, \dots, s_j \mathbf{e}_p), j=\{1,2\}	\qquad	
	\mathbf{r}^+: \mathbf{r}^+[k,\#(i,j)] :=&  \mathbf{e}_i[k] + \mathbf{e}_j[k], \\
	\mathbf{r}^-: \mathbf{r}^-[k,\#(i,j)] :=&  \mathbf{e}_i[k] - \mathbf{e}_j[k], \\
	\text{ for all } k=1, \dots, p; \text{ } &i=2, \dots, p \text{ and } j=1, \dots, i-1
	\end{split}
	\end{align}
	Both $\mathbf{r}^+$ and $\mathbf{r}^-$ have dimensions $p \times \frac{p(p-1)}{2}$, while those of vectors $\mathbf{t}^e_j$ are $p \times p$. These four matrices are used to compose $p \times (p^2+p)$-dimensional matrix $\tilde {\mathbf{t}}$:
	\begin{equation}
	\label{def:t_tilde}
	\tilde{\mathbf{t}} := (\mathbf{t}^e_1, \mathbf{t}^e_2, \mathbf{r}^+, \mathbf{r}^-)
	\end{equation}
	We use $\tilde{\mathbf{t}}$ to construct a statistic $\hat {\boldsymbol{\theta}}_n$ for which CLT (\ref{feuer_CLT_parts_ecf}) was shown. Specifically, we extract empirical characteristic functions using its columns:
	\begin{align}
	\hat {\boldsymbol{\theta}}_n = \begin{pmatrix} 
	\operatorname{Re}[\hat{\phi}_n(\tilde{\mathbf{t}}_1)] \\
	\operatorname{Re}[\hat{\phi}_n(\tilde{\mathbf{t}}_2)] \\
	\vdots \\
	\operatorname{Re}[\hat{\phi}_n(\tilde{\mathbf{t}}_{p^2+p})]\\
	\operatorname{Im}[\hat{\phi}_n(\tilde{\mathbf{t}}_1)] \\
	\operatorname{Im}[\hat{\phi}_n(\tilde{\mathbf{t}}_2)] \\
	\vdots \\
	\operatorname{Im}[\hat{\phi}_n(\tilde{\mathbf{t}}_{p^2+p})]\\
	\end{pmatrix} \quad
	{\boldsymbol{\theta}}_0 = \begin{pmatrix} 
	\operatorname{Re}[\phi(\tilde{\mathbf{t}}_1)] \\
	\operatorname{Re}[\phi(\tilde{\mathbf{t}}_2)] \\
	\vdots \\
	\operatorname{Re}[\phi(\tilde{\mathbf{t}}_{p^2+p})] \\
	\operatorname{Im}[\phi(\tilde{\mathbf{t}}_1)] \\
	\operatorname{Im}[\phi(\tilde{\mathbf{t}}_2)] \\
	\vdots \\
	\operatorname{Im}[\phi(\tilde{\mathbf{t}}_{p^2+p})] \\
	\end{pmatrix}
	\end{align}
	Application of the multidimensional delta method to function $g: R^{2p(p+1)} \to R^{1+p/2+p^2/2}$, such that $g(\cdot)=\big( \hat\alpha^{mult}(\cdot),\hat{\boldsymbol\Sigma}_{d}(\cdot), \hat{\boldsymbol\Sigma}_{nd}(\cdot)  \big)$ yields 
	\begin{equation}
	\sqrt{n} [g(\hat {\boldsymbol{\theta}}_n) - g(\boldsymbol{\theta}_0)] \stackrel{d}{\to} \mathbb{N} \big(0, \mathbf{G}({\boldsymbol{\theta}}_0)\mathbf{\Omega} \mathbf{G}^T({\boldsymbol{\theta}}_0)\big),
	\end{equation}
	where $\mathbf{G}$ is the Jacobian of $g$. 
	In what follows we will need partial derivatives shown below:
	\begin{align}
	\begin{split}
	\frac{\partial |\phi(\mathbf{t}_i)|}{\partial \operatorname{Re}[\phi(\mathbf{t}_i)]} = \frac{\operatorname{Re}[\phi(\mathbf{t}_i)]}{|\phi(\mathbf{t}_i)|}; \quad \frac{\partial |\phi(\mathbf{t}_i)|}{\partial \operatorname{Im}[\phi(\mathbf{t}_i)]} = \frac{\operatorname{Im}[\phi(\mathbf{t}_i)]}{|\phi(\mathbf{t}_i)|}
	\end{split}
	\end{align}
	
	\paragraph{The Jacobian of $g$.}  
	\textit{$\alpha$-related component.}
	Since the first component of the function depends only on the first $2p$ entries of ${\boldsymbol{\theta}}$, columns of the first row of $\boldsymbol{G}$ with indexes $2p+1,\dots, 2p(p+1)$ are all zeros. Therefore, it is more convenient to work with $\mathbf{t}^e_j$ instead of the whole vector $\tilde{\mathbf{t}}$. We take $\hat \alpha_s (k)$ from the representation  (\ref{def:alpha_s}) to obtain the following derivatives:
	\begin{align}
	\begin{split}
	\frac{\partial \hat\alpha_s (k)}{\partial|\hat \phi_n(s_1\mathbf{e}_{k})|}  &= 
	\frac{1}{ \log \big|\frac{s_{1}}{s_{2}}\big|} \frac{\log|\hat \phi_n(s_2\mathbf{e}_{k})|}{\log|\hat \phi_n(s_1\mathbf{e}_{k})|} \frac{1}{\log|\hat \phi_n(s_2\mathbf{e}_{k})|} \frac{1}{|\hat \phi_n(s_1\mathbf{e}_{k})|}	 \\
	\frac{\partial \hat\alpha_s (k)}{\partial \rm{Re}[\hat \phi_n(s_1\mathbf{e}_{k})]} = 
	\frac{\partial \hat\alpha_s (k)}{\partial|\hat \phi_n(s_1\mathbf{e}_{k})|} \frac{\partial|\hat \phi_n(s_1\mathbf{e}_{k})|}{\partial \rm{Re}[\hat \phi_n(s_1\mathbf{e}_{k})]} &= 
	\frac{1}{\log \big|\frac{s_1}{s_2}\big|} \frac{1}{\log|\hat \phi_n(s_1\mathbf{e}_{k})|} \frac{\rm{Re}[\hat \phi_n(s_1\mathbf{e}_{k})]}{|\hat \phi_n(s_1\mathbf{e}_{k})|^2} \\
	\text{Analogously,} \quad \frac{\partial \hat\alpha_s (k)}{\partial \rm{Im}[\hat \phi_n(s_1\mathbf{e}_{k})]} &=\frac{1}{\log \big|\frac{s_1}{s_2}\big|} \frac{1}{\log|\hat \phi_n(s_1\mathbf{e}_{k})|} \frac{\rm{Im}[\hat \phi_n(s_1\mathbf{e}_{k})]}{|\hat \phi_n(s_1\mathbf{e}_{k})|^2}
	\end{split}
	\end{align}
	\begin{align}
	\begin{split}
	\frac{\partial \hat\alpha_s (k)}{\partial|\hat \phi_n(s_2\mathbf{e}_{k})|}  = 
	\frac{1}{ \log \big|\frac{s_1}{s_2}\big|} \cdot \frac{\log|\hat \phi_n(s_2\mathbf{e}_{k})|}{\log|\hat \phi_n(s_1\mathbf{e}_{k})|} \cdot \log|\hat \phi_n(s_1\mathbf{e}_{k})| \cdot(-1) \cdot \frac{1}{\{\log |\hat \phi_n(s_2\mathbf{e}_{k})|\}^2} \frac{1}{|\hat \phi_n(s_2\mathbf{e}_{k})|} = \\-\left[ \log \left|\frac{s_1}{s_2} \right| \cdot \log|\hat \phi_n(s_2\mathbf{e}_{k})| \cdot |\hat \phi_n(s_2\mathbf{e}_{k})| \right]^{-1}	 \\
	\frac{\partial \hat\alpha_s (k)}{\partial \rm{Re}[\hat \phi_n(s_2\mathbf{e}_{k})]} = 
	\frac{\partial \hat\alpha_s (k)}{\partial|\hat \phi_n(s_2\mathbf{e}_{k})|} \frac{\partial|\hat \phi_n(s_2\mathbf{e}_{k})|}{\partial \rm{Re}[\hat \phi_n(s_2\mathbf{e}_{k})]} = 
	-\left[ \log \left|\frac{s_1}{s_2} \right| \cdot \log|\hat \phi_n(s_2\mathbf{e}_{k})| \right]^{-1} \cdot \frac{\rm{Re}[\hat \phi_n(s_2\mathbf{e}_{k})]}{|\hat \phi_n(s_2\mathbf{e}_{k})|^2} \\
	\text{Analogously,} \quad \frac{\partial \hat\alpha_s (k)}{\partial \rm{Im}[\hat \phi_n(s_2\mathbf{e}_{k})]} =-\left[ \log \left|\frac{s_1}{s_2} \right| \cdot \log|\hat \phi_n(s_2\mathbf{e}_{k})| \right]^{-1} \cdot \frac{\rm{Im}[\hat \phi_n(s_2\mathbf{e}_{k})]}{|\hat \phi_n(s_2\mathbf{e}_{k})|^2}
	\end{split}
	\end{align}
	Next, we use the above formulas to compute the first row of Jacobian $\boldsymbol{G}$. From (\ref{def:alpha_m}),
	\begin{align}
	\begin{split}
	\frac{\partial \hat\alpha^{mult}}{\partial \rm{Re} \big(\hat \phi_n (t_1^e[k])\big)} = \frac{\partial \hat\alpha_s(k)}{\partial \rm{Re} \big(\hat \phi_n (t_1^e[k])\big)} = (\boldsymbol{A}_s^1 \odot \boldsymbol{R}_{\phi}^1)[k] \\		
	\frac{\partial \hat\alpha^{mult}}{\partial \rm{Im} \big(\hat \phi_n (t_1^e[k])\big)} = \frac{\partial \hat\alpha_s(k)}{\partial \rm{Im} \big(\hat \phi_n (t_1^e[k])\big)} = (\boldsymbol{A}_s^1 \odot \boldsymbol{Im}_{\phi}^1)[k] 
	\end{split}
	\end{align}    
	\begin{equation}
	\label{def:ARIm_1}
	\boldsymbol{A}_s^1 [k]:=\left[\log \left|\frac{s_1}{s_2}\right| \cdot \log|\phi(s_1\mathbf{e}_{k})| \cdot |\phi(s_1\mathbf{e}_{k})| \right]^{-1} \quad \boldsymbol{R}_{\phi}^1 [k]:=\frac{\rm{Re}[\phi(s_1\mathbf{e}_{k})]}{|\phi(s_1\mathbf{e}_{k})|} \quad \boldsymbol{Im}^1_{\phi}[k]:= \frac{\rm{Im}[\phi(s_1\mathbf{e}_{k})]}{|\phi(s_1\mathbf{e}_{k})|}
	\end{equation}
	\begin{equation}
	\begin{split}
	\frac{\partial \hat\alpha^{mult}}{\partial \rm{Re} \big(\hat \phi_n (t_2^e[k])\big)} = \frac{\partial \hat\alpha_s(k)}{\partial \rm{Re} \big(\hat \phi_n (t_2^e[k])\big)} = (\boldsymbol{A}_s^2 \odot \boldsymbol{R}_{\phi}^2)[k] \\		
	\frac{\partial \hat\alpha^{mult}}{\partial \rm{Im} \big(\hat \phi_n (t_2^e[k])\big)} = \frac{\partial \hat\alpha_s(k)}{\partial \rm{Im} \big(\hat \phi_n (t_2^e[k])\big)} = (\boldsymbol{A}_s^2 \odot \boldsymbol{Im}_{\phi}^2)[k] 
	\end{split}
	\end{equation}     
	\begin{equation}
	\boldsymbol{A}_s^2 [k]:=-\left[ \log \left|\frac{s_1}{s_2} \right| \cdot \log|\phi(s_2\mathbf{e}_{k})| \cdot |\phi(s_2\mathbf{e}_{k})| \right]^{-1} \quad \boldsymbol{R}_{\phi}^2 [k]:=\frac{\rm{Re}[\phi(s_2\mathbf{e}_{k})]}{|\phi(s_2\mathbf{e}_{k})|} \quad \boldsymbol{Im}^2_{\phi}[k]:=\frac{\rm{Im}[\phi(s_2\mathbf{e}_{k})]}{|\phi(s_2\mathbf{e}_{k})|}
	\end{equation}    
	\textit{$\boldsymbol\Sigma_{ii}$-related components.}    
	\begin{equation}\label{eqn:quad_vol_est}
	\hat{\boldsymbol\Sigma}_d [j] = [\hat{\boldsymbol\Sigma}]_{jj} = \frac{2(-\operatorname{log}|\hat \phi_n(s_1\mathbf{e}_j)|)^{2/\hat \alpha^{mult}}}{{s^2_1}} =    \frac{2\exp \left[\log\{-\operatorname{log}|\hat \phi_n(s_1\mathbf{e}_j)|\}\frac{2}{\hat \alpha^{mult}}\right]}{{s^2_1}}, \quad j \leq p.
	\end{equation} 
	\begin{equation}
	\begin{split}
	\text{The diagonal elements,} \quad
	\frac{\partial[\hat{\boldsymbol\Sigma}]_{jj}}{\partial |\hat \phi_n (s_1 \mathbf{e}_j)|} = 2 s^{-2}_1 \exp \left[\log\{-\operatorname{log}|\hat \phi_n(s_1\mathbf{e}_j)|\}\frac{2}{\hat \alpha^{mult}}\right] \times
	\\
	\left[ \frac{1}{-\log|\hat \phi_n(s_1\mathbf{e}_j)|} \cdot \frac{-1}{|\hat \phi_n(s_1\mathbf{e}_j)|} \cdot \frac{2}{\hat \alpha^{mult}} + \log\{ -\log|\hat \phi_n(s_1 \mathbf{e}_j)| \} \cdot \frac{-2}{(\hat \alpha^{mult})^2} \cdot \frac{1}{\log \left|\frac{s_1}{s_2} \right| \cdot \log|\hat\phi_n(s_1\mathbf{e}_j)| \cdot |\hat\phi_n(s_1\mathbf{e}_j)|} \right]
	\end{split}
	\end{equation}
	After plugging-in $\phi$'s in place of $\hat \phi_n$'s and using  (\ref{eqn:quad_vol_est}), we reduce the expression:
	\begin{equation}
	\begin{split}
	&\frac{\partial[\hat{\boldsymbol\Sigma}]_{jj}}{\partial |\hat \phi_n (s_1 \mathbf{e}_j)|} = \\
	&\frac{2}{\alpha} \boldsymbol\Sigma_{jj} \left[(\log| \phi(s_1\mathbf{e}_j)| \cdot | \phi(s_1\mathbf{e}_j)|)^{-1} + \frac{\alpha}{2} \log\left(\frac{s_1^2 \boldsymbol\Sigma_{jj}}{2}\right) \cdot \frac{-1}{\alpha} \cdot \left(\log\left|\frac{s_1}{s_2}\right| \cdot \log| \phi(s_1\mathbf{e}_j)| \cdot | \phi(s_1\mathbf{e}_j)|\right)^{-1} \right] = \\
	&\frac{2\boldsymbol\Sigma_{jj}}{\alpha} \big[\log| \phi(s_1\mathbf{e}_j)| \cdot | \phi(s_1\mathbf{e}_j)|\big]^{-1} \left(1 - \frac{\log\left(\frac{s_1^2 \boldsymbol\Sigma_{jj}}{2}\right)}{2\log\left|\frac{s_1}{s_2}\right|}\right)
	\end{split}
	\end{equation}
	\begin{equation}
	\begin{split}
	&j \neq k, \quad \frac{\partial[\hat{\boldsymbol\Sigma}]_{jj}}{\partial |\hat \phi_n (s_1 \mathbf{e}_k)|} = 2s^{-2}_1 \exp \left[\log\{-\operatorname{log}|\hat \phi_n(s_1\mathbf{e}_j)|\}\frac{2}{\hat \alpha^{mult}}\right] \times
	\\
	&\left[ \log\{ -\log|\hat \phi_n(s_1 \mathbf{e}_j)| \} \cdot \frac{-2}{[\hat \alpha^{mult}]^2} \cdot \frac{1}{\log\left|\frac{s_1}{s_2}\right| \cdot \log|\hat\phi_n(s_1\mathbf{e}_k)| \cdot |\hat\phi_n(s_1\mathbf{e}_k)|} \right]
	\end{split}
	\end{equation}
	which, after evaluating at point ${\boldsymbol{\theta}}_0$ becomes
	\begin{equation}
	\begin{split}
	\frac{\partial[\hat{\boldsymbol\Sigma}]_{jj}}{\partial |\hat \phi_n (s_1 \mathbf{e}_k)|} = \boldsymbol\Sigma_{jj} \cdot \frac{\alpha}{2} \cdot \log \left( \frac{s^2_1}{2} \boldsymbol\Sigma_{jj} \right) \cdot \frac{-2}{\alpha} \cdot \left( \log\left|\frac{s_1}{s_2}\right| \cdot \log|\phi(s_1\mathbf{e}_k)| \cdot |\phi(s_1\mathbf{e}_k)| \right)^{-1} =\\
	-\boldsymbol\Sigma_{jj} \cdot \log \left( \frac{s^2_1}{2} \boldsymbol\Sigma_{jj} \right) \cdot
	\left( \log\left|\frac{s_1}{s_2}\right| \cdot \log|\phi(s_1\mathbf{e}_k)| \cdot |\phi(s_1\mathbf{e}_k)| \right)^{-1}
	\end{split}
	\end{equation}
	\begin{equation}
	\begin{split}
	\forall j=1, \dots, p; k=1, \dots, p, \quad \frac{\partial[\hat{\boldsymbol\Sigma}]_{jj}}{\partial |\hat \phi_n (s_2 \mathbf{e}_k)|} = 2s^{-2}_1 \exp \left[\log\{-\operatorname{log}|\hat \phi_n(s_1\mathbf{e}_j)|\}\frac{2}{\hat \alpha^{mult}}\right] \cdot
	\\
	\log\{ -\log|\hat \phi_n(s_1 \mathbf{e}_j)| \} \cdot \frac{-2}{[\hat \alpha^{mult}]^2} \cdot \left[- \log \left|\frac{s_1}{s_2}\right| \cdot \log|\hat\phi_n(s_2\mathbf{e}_k)| \cdot |\hat\phi_n(s_2\mathbf{e}_k)| \right]^{-1}
	\end{split}
	\end{equation}
	, after evaluating at point ${\boldsymbol{\theta}}_0$ becomes
	\begin{equation}
	\begin{split}
	\frac{\partial[\hat{\boldsymbol\Sigma}]_{jj}}{\partial |\hat \phi_n (s_2 \mathbf{e}_k)|} = \boldsymbol\Sigma_{jj} \cdot \frac{\alpha}{2} \cdot \log \left( \frac{s^2_1}{2} \boldsymbol\Sigma_{jj} \right) \cdot \frac{-2}{\alpha^2} \cdot \left(- \log\left|\frac{s_1}{s_2}\right| \cdot \log|\phi(s_2\mathbf{e}_k)| \cdot |\phi(s_2\mathbf{e}_k)| \right)^{-1} =\\
	\boldsymbol\Sigma_{jj} \cdot \log \left( \frac{s^2_1}{2} \boldsymbol\Sigma_{jj} \right) \cdot
	\left( \alpha \cdot \log\left|\frac{s_1}{s_2}\right| \cdot \log|\phi(s_2\mathbf{e}_k)| \cdot |\phi(s_2\mathbf{e}_k)| \right)^{-1}
	\end{split}
	\end{equation}   
	\textit{$\boldsymbol\Sigma_{ij}$-related components.}     
	After rewriting the second formula in (\ref{eqn:matrix_estims}) in the form
	\begin{equation}
	\begin{split}
	[\hat{\boldsymbol\Sigma}]_{ij} = \frac{1}{2} \exp \left[\log\left(-\log \left\{|\hat \phi_n(\mathbf{e}_i+\mathbf{e}_j)|\right\}\right)\frac{2}{\hat \alpha^{mult}}\right] - \frac{1}{2} \exp \left[\log\left(-\log \left\{|\hat \phi_n(\mathbf{e}_i-\mathbf{e}_j)|\right\}\right)\frac{2}{\hat \alpha^{mult}}\right]
	\end{split}
	\end{equation}
	we write its derivative as 
	\begin{equation}
	\begin{split}
	&\frac{\partial [\hat{\boldsymbol\Sigma}]_{ij}}{\partial |\hat \phi_n(s_1\mathbf{e}_k)|} = 
	\\
	&\frac{1}{2} \exp \left[\log\left(-\log \left\{|\hat \phi_n(\mathbf{e}_i+\mathbf{e}_j)|\right\}\right)\frac{2}{\hat \alpha^{mult}}\right] \times \log\left(-\log \left\{|\hat \phi_n(\mathbf{e}_i+\mathbf{e}_j)|\right\}\right)\frac{2}{(\hat\alpha^{mult})^2} \times \frac{\partial \hat \alpha^{mult}}{\partial|\hat \phi_n(s_1\mathbf{e}_k)|} - 
	\\
	&\frac{1}{2} \exp \left[\log\left(-\log \left\{|\hat \phi_n(\mathbf{e}_i-\mathbf{e}_j)|\right\}\right)\frac{2}{\hat \alpha^{mult}}\right] \times \log\left(-\log \left\{|\hat \phi_n(\mathbf{e}_i-\mathbf{e}_j)|\right\}\right)\frac{2}{(\hat\alpha^{mult})^2} \times \frac{\partial \hat \alpha^{mult}}{\partial|\hat \phi_n(s_1\mathbf{e}_k)|}
	\\
	\end{split}
	\end{equation} 
	, after evaluating at point ${\boldsymbol{\theta}}_0$ becomes
	\begin{equation}
	\begin{split}
	&\frac{\partial [\hat{\boldsymbol\Sigma}]_{ij}}{\partial |\hat \phi_n(s_1\mathbf{e}_k)|} = 
	\Big[\frac{1}{2} \cdot \Big(\frac{1}{2}\boldsymbol\Sigma_{ii} + \boldsymbol\Sigma_{ij} + \frac{1}{2}\boldsymbol\Sigma_{jj}\Big) \cdot \frac{\alpha}{2} \log \Big(\frac{1}{2}\boldsymbol\Sigma_{ii} + \boldsymbol\Sigma_{ij} + \frac{1}{2}\boldsymbol\Sigma_{jj}\Big)- 
	\\
	&\frac{1}{2} \cdot \Big(\frac{1}{2}\boldsymbol\Sigma_{ii} - \boldsymbol\Sigma_{ij} + \frac{1}{2}\boldsymbol\Sigma_{jj}\Big) \cdot \frac{\alpha}{2} \log \Big(\frac{1}{2}\boldsymbol\Sigma_{ii} - \boldsymbol\Sigma_{ij} + \frac{1}{2}\boldsymbol\Sigma_{jj}\Big)\Big] \cdot
	\\
	&\frac{2}{\alpha^2} \cdot \left( \log\left|\frac{s_1}{s_2}\right| \cdot \log|\phi(s_1\mathbf{e}_k)| \cdot |\phi(s_1\mathbf{e}_k)| \right)^{-1}
	\end{split}
	\end{equation}
	After simplification, 
	\begin{equation}
	\begin{split}
	&\frac{\partial [\hat{\boldsymbol\Sigma}]_{ij}}{\partial |\hat \phi_n(s_1\mathbf{e}_k)|} = 
	\Big[ \Big(\frac{1}{2}\boldsymbol\Sigma_{ii} + \boldsymbol\Sigma_{ij} + \frac{1}{2}\boldsymbol\Sigma_{jj}\Big) \cdot  \log \Big(\frac{1}{2}\boldsymbol\Sigma_{ii} + \boldsymbol\Sigma_{ij} + \frac{1}{2}\boldsymbol\Sigma_{jj}\Big)- 
	\\
	&\Big(\frac{1}{2}\boldsymbol\Sigma_{ii} - \boldsymbol\Sigma_{ij} + \frac{1}{2}\boldsymbol\Sigma_{jj}\Big) \cdot  \log \Big(\frac{1}{2}\boldsymbol\Sigma_{ii} - \boldsymbol\Sigma_{ij} + \frac{1}{2}\boldsymbol\Sigma_{jj}\Big)\Big] \cdot
	\left(2 \alpha \cdot \log\left|\frac{s_1}{s_2}\right| \cdot \log|\phi(s_1\mathbf{e}_k)| \cdot |\phi(s_1\mathbf{e}_k)| \right)^{-1}
	\end{split}
	\end{equation}    
	Derivation of $\frac{\partial [\hat{\boldsymbol\Sigma}]_{ij}}{\partial |\hat \phi_n(s_2\mathbf{e}_k)|}$ is analogous:
	\begin{equation}
	\begin{split}
	&\frac{\partial [\hat{\boldsymbol\Sigma}]_{ij}}{\partial |\hat \phi_n(s_2\mathbf{e}_k)|} = 
	\Big[\frac{1}{2} \cdot \Big(\frac{1}{2}\boldsymbol\Sigma_{ii} + \boldsymbol\Sigma_{ij} + \frac{1}{2}\boldsymbol\Sigma_{jj}\Big) \cdot \frac{\alpha}{2} \log \Big(\frac{1}{2}\boldsymbol\Sigma_{ii} + \boldsymbol\Sigma_{ij} + \frac{1}{2}\boldsymbol\Sigma_{jj}\Big)- 
	\\
	&\frac{1}{2} \cdot \Big(\frac{1}{2}\boldsymbol\Sigma_{ii} - \boldsymbol\Sigma_{ij} + \frac{1}{2}\boldsymbol\Sigma_{jj}\Big) \cdot \frac{\alpha}{2} \log \Big(\frac{1}{2}\boldsymbol\Sigma_{ii} - \boldsymbol\Sigma_{ij} + \frac{1}{2}\boldsymbol\Sigma_{jj}\Big)\Big] \cdot
	\\
	&\frac{2}{\alpha^2} \cdot \left(- \log\left|\frac{s_1}{s_2}\right| \cdot \log|\phi(s_2\mathbf{e}_k)| \cdot |\phi(s_2\mathbf{e}_k)| \right)^{-1}
	\end{split}
	\end{equation}  
	\begin{equation}
	\begin{split}
	&\frac{\partial [\hat{\boldsymbol\Sigma}]_{ij}}{\partial |\hat \phi_n(s_2\mathbf{e}_k)|} = 
	\Big[ \Big(\frac{1}{2}\boldsymbol\Sigma_{ii} + \boldsymbol\Sigma_{ij} + \frac{1}{2}\boldsymbol\Sigma_{jj}\Big) \cdot  \log \Big(\frac{1}{2}\boldsymbol\Sigma_{ii} + \boldsymbol\Sigma_{ij} + \frac{1}{2}\boldsymbol\Sigma_{jj}\Big)- 
	\\
	&\Big(\frac{1}{2}\boldsymbol\Sigma_{ii} - \boldsymbol\Sigma_{ij} + \frac{1}{2}\boldsymbol\Sigma_{jj}\Big) \cdot  \log \Big(\frac{1}{2}\boldsymbol\Sigma_{ii} - \boldsymbol\Sigma_{ij} + \frac{1}{2}\boldsymbol\Sigma_{jj}\Big)\Big] \cdot
	\\
	&\left(-2\alpha \cdot \log\left|\frac{s_1}{s_2}\right| \cdot \log|\phi(s_2\mathbf{e}_k)| \cdot |\phi(s_2\mathbf{e}_k)| \right)^{-1}
	\end{split}
	\end{equation}  
	\begin{equation}
	\begin{split}
	&\frac{\partial [\hat{\boldsymbol\Sigma}]_{ij}}{\partial |\hat \phi_n(\mathbf{e}_i + \mathbf{e}_j)|} = 
	\\
	&\frac{1}{2} \exp \left[\log\left(-\log \left\{|\hat \phi_n(\mathbf{e}_i+\mathbf{e}_j)|\right\}\right)\frac{2}{\hat \alpha^{mult}}\right] \cdot \frac{2}{\hat \alpha^{mult}} \cdot \frac{1}{-\log \left\{|\hat \phi_n(\mathbf{e}_i+\mathbf{e}_j)|\right\}} \cdot \frac{-1}{|\hat \phi_n(\mathbf{e}_i+\mathbf{e}_j)|}  =
	\\
	&\left(-\log \left\{|\hat \phi_n(\mathbf{e}_i+\mathbf{e}_j)|\right\}\right)^{\frac{2}{\hat \alpha^{mult}}} \cdot \frac{1}{\hat \alpha^{mult}} \cdot \frac{1}{\log \left\{|\hat \phi_n(\mathbf{e}_i+\mathbf{e}_j)|\right\}} \cdot \frac{1}{|\hat \phi_n(\mathbf{e}_i+\mathbf{e}_j)|} 
	\\
	&\frac{\partial [\hat{\boldsymbol\Sigma}]_{ij}}{\partial |\hat \phi_n(\mathbf{e}_i - \mathbf{e}_j)|} = 
	-\exp \left[\log\left(-\log \left\{|\hat \phi_n(\mathbf{e}_i-\mathbf{e}_j)|\right\}\right)\frac{2}{\hat \alpha^{mult}}\right] \cdot \frac{1}{\hat \alpha^{mult}} \cdot \frac{1}{-\log \left\{|\hat \phi_n(\mathbf{e}_i-\mathbf{e}_j)|\right\}} \cdot \frac{-1}{|\hat \phi_n(\mathbf{e}_i-\mathbf{e}_j)|} =
	\\
	&-\left(-\log \left\{|\hat \phi_n(\mathbf{e}_i-\mathbf{e}_j)|\right\}\right)^{\frac{2}{\hat \alpha^{mult}}} \cdot \frac{1}{\hat \alpha^{mult}} \cdot \frac{1}{\log \left\{|\hat \phi_n(\mathbf{e}_i-\mathbf{e}_j)|\right\}} \cdot \frac{1}{|\hat \phi_n(\mathbf{e}_i-\mathbf{e}_j)|}      
	\end{split}
	\end{equation}
	\begin{equation}
	\frac{\partial [\hat{\boldsymbol\Sigma}]_{ij}}{\partial | \phi_n(\mathbf{e}_i + \mathbf{e}_j)|} =
	\frac{- \left(\frac{1}{2} \boldsymbol\Sigma_{ii} + \boldsymbol\Sigma_{ij} + \frac{1}{2} \boldsymbol\Sigma_{jj}\right)^{(1-\frac{\alpha}{2})}}{\alpha \cdot |\phi(\mathbf{e}_i+\mathbf{e}_j)|};  
	\qquad
	\frac{\partial [\hat{\boldsymbol\Sigma}]_{ij}}{\partial |\hat \phi_n(\mathbf{e}_i - \mathbf{e}_j)|} =
	\frac{\left(\frac{1}{2} \boldsymbol\Sigma_{ii} - \boldsymbol\Sigma_{ij} + \frac{1}{2} \boldsymbol\Sigma_{jj}\right)^{(1-\frac{\alpha}{2})}}{\alpha \cdot |\phi(\mathbf{e}_i-\mathbf{e}_j)|}          
	\end{equation}
	
	\begin{equation}
	\begin{split}
	\frac{\partial \hat{\boldsymbol\Sigma}_{nd}[\#]}{\partial \rm{Re}\left(\hat \phi_n(\mathbf{e}_i+\mathbf{e}_j)\right)} = \frac{\partial \hat{\boldsymbol\Sigma}_{ij}}{\partial |\hat \phi_n(\mathbf{e}_i+\mathbf{e}_j)|} \frac{\partial |\hat \phi_n(\mathbf{e}_i+\mathbf{e}_j)|}{\partial \rm{Re}\left(\hat \phi_n(\mathbf{e}_i+\mathbf{e}_j)\right)} =
	\frac{\partial \hat{\boldsymbol\Sigma}_{ij}}{\partial |\hat \phi_n(\mathbf{e}_i+\mathbf{e}_j)|} \frac{ \rm{Re}\left(\hat \phi_n(\mathbf{e}_i+\mathbf{e}_j)\right)}{ |\hat \phi_n(\mathbf{e}_i+\mathbf{e}_j)| }        
	\\    	 
	\end{split}
	\end{equation}    
	
\end{proof}



\bibliographystyle{unsrt}
\bibliography{biblio}

\end{document}